\def\year{2022}\relax
\documentclass[letterpaper]{article} 
\usepackage{aaai22}  
\usepackage{times}  
\usepackage{helvet}  
\usepackage{courier}  
\usepackage[hyphens]{url}  
\usepackage{graphicx} 
\usepackage{natbib}  
\usepackage{caption} 
\DeclareCaptionStyle{ruled}{labelfont=normalfont,labelsep=colon,strut=off} 
\usepackage{algorithm}
\usepackage{algorithmic}

\usepackage{newfloat}
\usepackage{listings}
\floatstyle{ruled}
\newfloat{listing}{tb}{lst}{}
\floatname{listing}{Listing}
\usepackage{amsmath}
\usepackage{amsthm}
\usepackage{amsfonts}
\usepackage{lipsum}
\usepackage{comment}
\usepackage{subfigure}
\newcommand{\R}{\mathbb{R}}
\newcommand{\fnn}{f_\mathrm{NN}}
\newcommand{\hnn}{H_\mathrm{NN}}
\newcommand{\unn}{u_\mathrm{NN}}
\newtheorem{thm}{Theorem}
\newtheorem{lem}{Lemma}
\newtheorem{dfn}{Definition}

\renewcommand{\d}{\mathrm{d}}
\newtheorem{rem}{Remark}
\newcommand{\NN}{\mathrm{NN}}

\title{KAM Theory Meets Statistical Learning Theory: \\ Hamiltonian Neural Networks with Non-Zero Training Loss
}
\author{
    Yuhan Chen, \textsuperscript{\rm 1}
    Takashi Matsubara, \textsuperscript{\rm 2}
    Takaharu Yaguchi\textsuperscript{\rm 1}
    \equalcontrib,
}
\affiliations{
    \textsuperscript{\rm 1}Kobe University, 1--1 Rokkodai, Nada, Kobe, Hyogo, Japan 657--8501\\
    \textsuperscript{\rm 2}Osaka University, 1--3 Machikaneyama, Toyonaka, Osaka, Japan 560--8531\\
     193x226x@stu.kobe-u.ac.jp, matsubara@sys.es.osaka-u.ac.jp, yaguchi@pearl.kobe-u.ac.jp
}

\begin{document}

\maketitle

\begin{abstract}
	Many physical phenomena are described by Hamiltonian mechanics using an energy function (the Hamiltonian). Recently, the Hamiltonian neural network, which approximates the Hamiltonian as a neural network, and its extensions have attracted much attention. This is a very powerful method, but its use in theoretical studies remains limited. 
	In this study, by combining the statistical learning theory and Kolmogorov--Arnold--Moser (KAM) theory, we provide a theoretical analysis of the behavior of Hamiltonian neural networks when the learning error is not completely zero.
	A Hamiltonian neural network with non-zero errors can be considered as a perturbation from the true dynamics, and the perturbation theory of the Hamilton equation is widely known as the KAM theory. 
	To apply the KAM theory, we provide a generalization error bound for Hamiltonian neural networks by deriving an estimate of the covering number of the gradient of the multi-layer perceptron, which is the key ingredient of the model.
	This error bound gives 
	an $L^\infty$ bound on the Hamiltonian that is required in the application of the KAM theory. 
\end{abstract}

\section{Introduction}
Many physical phenomena are described by energy-based theories, such as Hamiltonian mechanics (e.g., \citet{Furihata2010}).
The governing equation of Hamiltonian mechanics is
\begin{align}\label{eq:target}
	\frac{\mathrm{d}u}{\mathrm{d}t} = S \frac{\partial H}{\partial u},
\end{align}
where $u: t \in \R \mapsto u(t) \in \R^N$, $H: u \in \R^N \mapsto H(u) \in \R$, and $S$ is a skew-symmetric matrix. 
$H$ represents the energy function of the system.
In recent years, there has been a lot of research on predicting the corresponding physical phenomena by learning the energy function $H$ in such equations with a neural network $\hnn$~(e.g., \citet{Chen2019-eo,Cranmer2020-zw,Greydanus2019-gy,Matsubara2019-ey,Zhong2019-rl}); however, to the best of our knowledge, theoretical analysis of such models has not been performed sufficiently, except for SympNet \citep{Jin2020-zs} for the Hamilton equation, where the universal approximation theorems for discrete-time neural network models are provided.

In this paper, we focus on theories of the properties of the most fundamental model, comprising Hamiltonian neural networks (HNNs)~\citep{Greydanus2019-gy} 
\begin{align} \label{eq:model}
	\frac{\d u}{\d t} = S \frac{\partial \hnn}{\partial u}
\end{align}
and their extensions in practical situations, where the learning error is not completely zero.
In this case, the trained model can be regarded as a perturbed Hamiltonian system due to the modeling error of the energy function.
In addition, $S$ is a general skew-symmetric matrix and hence \eqref{eq:model} can model Hamiltonian partial differential equations~\cite{Matsubara2019-ey}. 
\begin{figure*}[t]
\centering
\includegraphics[width=0.9\linewidth]{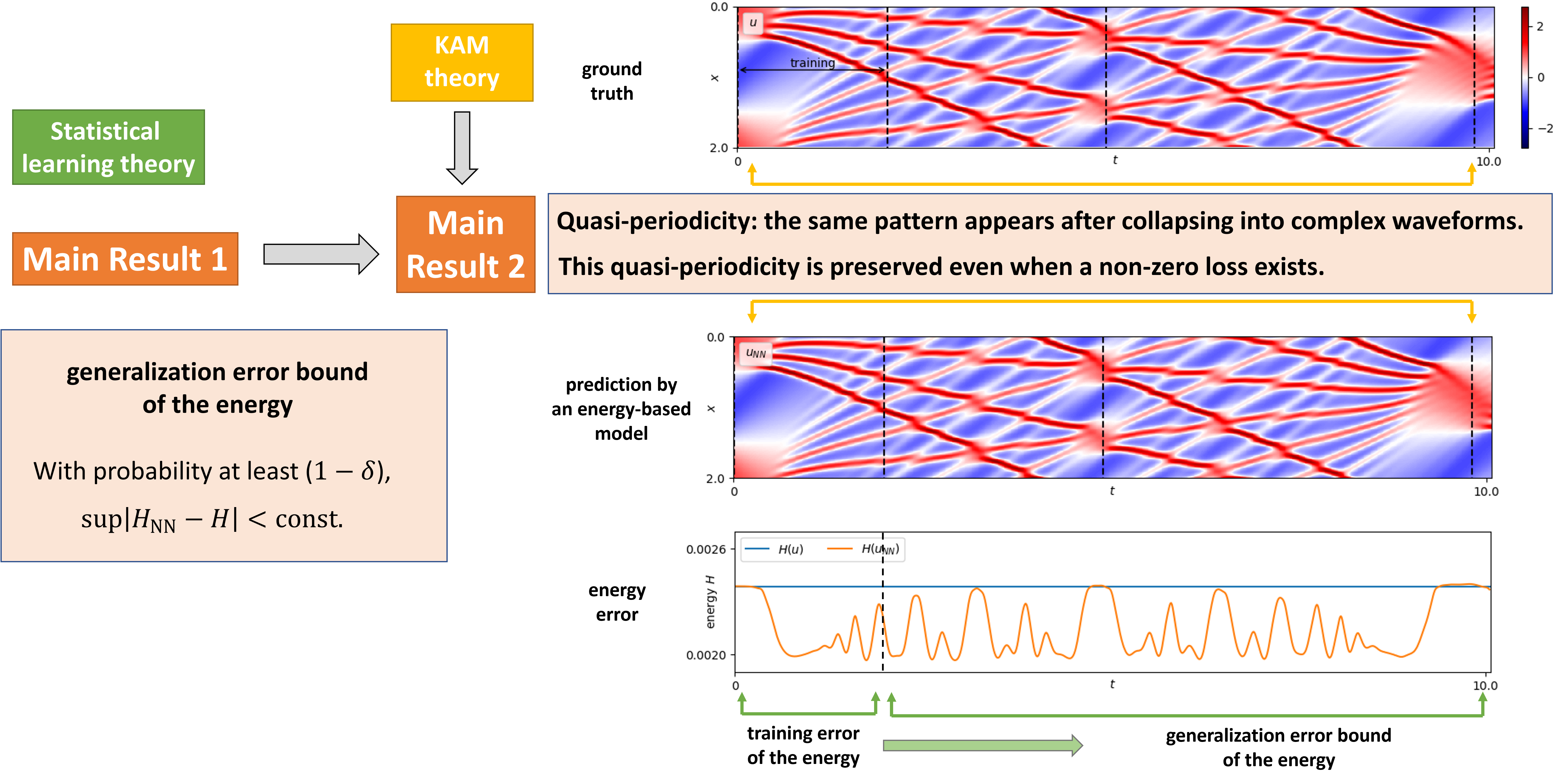}
\caption{Outline of our main theorem. The first main result is the generalization error of the energy function, which is proved by the statistical learning theory. By combining the first result and the KAM theorem, we prove that the quasi-periodic behaviors of target systems are preserved.} 
\label{fig:2} 
\end{figure*}

In mathematical physics, perturbed Hamiltonian systems are well studied.
For example, because whether the solar system will continue to exist in the future is of fundamental interest in astronomy, the stability of the solar system under perturbations is a very important issue that has been studied for a long time (e.g., \citet{Fejoz2013-ud, Laskar1996}).
The Kolmogorov--Arnold--Moser (KAM) theory gives an answer to questions of this type; essentially, periodic motions of such systems are stable under small perturbations.
The stability of periodic motions is of particular importance in science.
In addition to the stability of celestial systems, the recursive nature of physical phenomena is also of interest in physics.  For example, in the famous numerical experiments using the Korteweg--De Vries (KdV) equation by \citet{Zabusky1965-mn}, it was confirmed that the waveform given as the initial condition initially collapsed due to severe oscillations, and then returned to its original shape after a long time.
Whether such phenomena can be reproduced by deep learning models is an important problem that greatly affects the usefulness of deep physical models.

In this paper, we give an answer to this question by combining the KAM theory and statistical learning theory.
Because trained models are also perturbed Hamiltonian systems, it is expected that the periodic behaviors of the systems are preserved even if the loss function does not vanish completely provided it is sufficiently small.
However, this expectation cannot be proved in a straightforward way, because the application of the KAM theory requires that {\em the energy function be close enough to that of the true dynamics in the whole phase space.}
Noticing that the error of the energy function in the whole space is essentially the generalization error, we overcome this difficulty by combining the KAM theory with statistical learning theory.
This illustrates that combinations of statistical learning theory and dynamical systems theory can lead to powerful results.
Indeed, combinations of this kind may be applicable to other stability results in dynamical system contexts, for example, the stability of solitary wave solutions.

Importantly, for the neural network models to be close to the true dynamics, we need a universal approximation theorem and also a generalization error bound. In this paper, we also provide such results for HNNs.  

Regarding the generalization error bound, because the derivative of a multi-layer perceptron is used in HNNs, a bound for the derivative is required. To this end, we estimated the covering number of the derivative of multi-layer perceptrons.
An $L^\infty$ bound on the error in the Hamiltonian is also provided, which is required for application of the KAM theory.

In addition, we show a universal approximation theorem for a model with the coordinate transformation 
\begin{align}\label{eq:transformed}
	\frac{\mathrm{d}x}{\mathrm{d}t}=(\frac{\partial u}{\partial x})^{-1} S (\frac{\partial u}{\partial x})^{-\top} \frac{\partial H}{\partial x}.
\end{align}
This model is indispensable in practice; to apply HNNs, the data are given in the canonical coordinate system because the equation of motion is in the form of the Hamilton equation \eqref{eq:target} only when the state variables are represented by the canonical coordinate system. However, this coordinate system depends on an unknown Lagrangian and hence the energy function. 
Hence, the coordinate system must be also learned from data by using, for example, neural networks.
In addition, this model also represents other energy-based physical models beyond the Hamilton equation. 
See the Appendix for details. 

The main contributions of this paper are as follows.
\begin{enumerate}
	\item \textbf{Combination of the KAM theory and statistical learning theory for HNNs with non-zero training loss to prove the existence of quasi-periodic behaviors (see Figure \ref{fig:2}).}
	
\item \textbf{Derivation of a generalization error bound for HNNs.} 
	\item \textbf{Development of a universal approximation theorem for HNNs and other energy-based physical models with coordinate transformations.}
\end{enumerate}

\section{Related work}\label{sec:2}
Many studies of neural network models for phenomena that can be modeled by energy-based equation \eqref{eq:target} have been put forward.
Among them, the most basic studies are neural ordinary differential equations \citep{Chen2018-yw} and HNNs \citep{Greydanus2019-gy}. In particular, extensions of HNNs have been intensively developed.

Describing them all is beyond the scope of this paper, but some examples are given here. In \citet{Toth2019-ix} HNNs were extended to latent variable models.
Other studies, such as \citet{ DiPietro2020-mj, Xiong2020-ci, Zhong2019-rl}, focused on the symplectic structure of the Hamilton equation.
For Noether's theorem,
which is a fundamental theorem in classical mechanics,
several studies 
\cite{Bharadwaj2020-fr, Bondesan2019-vm, Finzi2020-zh} developed methods related to symmetry and conservation laws.
In addition, a discrete-time model that preserves the energy behaviors was constructed in \citet{Matsubara2019-ey}.
In \citet{Galioto2020-xc}, HNNs were combined with a Bayesian approach.

Methods applied to the framework of classical mechanics other than Hamiltonian mechanics include those in \citet{Cranmer2020-zw, Desai2020-ar,noauthor_undated-ko}, which are methods for Lagrangian formalism. In \citet{Jin2020-wl}, reinforcement learning was applied to the variational principle.
A simplified model formed by introducing constraints was proposed in \citet{Finzi2020-qa}.
In \citet{Jin2020-qd}, HNNs were extended to the Poisson system, which is a wider class of mechanical equations.
There are also a number of proposals that integrate them with more advanced deep learning techniques, namely, graph networks \citep{Sanchez-Gonzalez2019-js}, recurrent neural networks \citep{Chen2019-eo}, and normalizing flows \citep{Li2020-rx}.
As an application-oriented approach, \citet{Feng2020-hd} designed a microscopic model for structural analysis.

However, as far as the authors know, there is no theoretical research other than the universal approximation theorems for Hamiltonian mechanics in SympNet \cite{Jin2020-zs}, in which a certain kind of neural network is shown to have universal approximation properties for symplectic maps. The difference between their results and ours is that (1) we analyze the behaviors of a HNN with non-zero training loss by a combination of the KAM theory and statistical learning theory, (2) we provide a generalization error bound for NHHs, and (3) the universal approximation theorems in \citet{Jin2020-zs} are for discrete-time models, while ours are for continuous-time models.

Meanwhile, as an existing energy-based model, Hopfield neural network is known. Both Hopfield neural network and Hamiltonian network are derived from energy-based theories, and their dynamics are described by (1).
Hamiltonian neural network is associated with a skew-symmetric matrix $S$ and is a model of an energy-preserving, continuous-time, and deterministic physics phenomenon.
Its output is the time-series of the state.
Hopfield network is associated with a negative definite matrix $S$ and exhibits a dynamics which is often energy-dissipating, discrete-time, and stochastic.
It is a machine learning tool rather than a physical model, and its equilibrium point is treated as its output.
Because their outputs are different, their theoretical properties should be discussed separately.

\section{Brief introduction to Hamiltonian systems and the KAM theory}




We briefly introduce some properties of Hamiltonian systems.

\begin{thm}[Darboux]
	By an appropriate coordinate transformation, the matrix $S$ is transformed into the normal form
	\begin{align*}
		\begin{pmatrix}
			O  & I \\
			-I & O
		\end{pmatrix}.
	\end{align*}
\end{thm}
\begin{dfn}
	The function $\omega: (v, w) \in \R^N \times \R^N \mapsto \omega(v, w) \in \R$
	\begin{align*}
		\omega(v, w) = v^\top S^{-1} w
	\end{align*}
	is called the symplectic form.
	Using the symplectic form associates a vector field $X_F$ with each function $F: \R^N \to \R$ by requiring
	\begin{align*}
		\omega(X_F, w) = \frac{\partial F}{\partial u} \cdot w \qquad \mbox{for all}\ w.
	\end{align*}
	For two functions $F, G$, the following operation
	is called the Poisson bracket:
	\begin{align}
		\{F, G\} = \omega(X_F, X_G).
	\end{align}
\end{dfn}

\begin{dfn}
	A Hamiltonian system for which the state variable is $N=2 M$ dimensional is integrable in the sense of Liouville if
	this Hamiltonian system has the first integrals (i.e., conserved quantities)
	$F_1, F_2, \ldots, F_M$ with $\nabla F_1(u), \nabla F_2(u), \ldots, \nabla F_M(u)$ independent at each $u$ and for all $i, j$:
	\begin{align*}
		\{F_i, F_j\} = 0.
	\end{align*}
\end{dfn}

For integrable systems, Theorem 2 is known.
\begin{thm}[Liouville--Arnold]\label{thm:LA}
	Suppose that for an integrable Hamiltonian system, constants $c_1,\ldots,c_M$ exist such that $K = \cap_{i=1}^M F_i^{-1}(c_i)$ is connected and compact.
	Then, there exists a neighborhood $\mathcal{N}$ comprising $K$, $\mathcal{U} \subset \R^n$  and a coordinate transform
	\begin{align}
		\phi: (\theta, J) \in \mathbb{T}^n \times \mathcal{U} \to \phi(\theta, J) \in \mathcal{N}
	\end{align}
	such that
	the transformed system is the Hamilton equation of which Hamiltonian $H \circ \phi$ depends only on $J$.
\end{thm}
The variables $J$ and $\theta$ are called action-angle variables.
Theorems 1 and 2 roughly mean that integrable Hamiltonian systems can be written in the following form:
\begin{align*}
	\frac{\d}{\d t}\begin{pmatrix}
		\theta \\ J
	\end{pmatrix}
	= \begin{pmatrix}
		O  & I \\
		-I & O
	\end{pmatrix}
	\begin{pmatrix}
		\frac{\partial \tilde{H}}{\partial \theta} \\[0.2em]
		\frac{\partial \tilde{H}}{\partial J}
	\end{pmatrix}.
\end{align*}
Further, because $\tilde{H}=H \circ \phi$ depends on $I$ only, it holds that
\begin{align*}
	\frac{\d}{\d t}\begin{pmatrix}
		\theta \\ J
	\end{pmatrix}
	= \begin{pmatrix}
		O  & I \\
		-I & O
	\end{pmatrix}
	\begin{pmatrix}
		0 \\
		\frac{\partial \tilde{H}}{\partial J}
	\end{pmatrix}
	=
	\begin{pmatrix}
		\frac{\partial \tilde{H}}{\partial J} \\
		0
	\end{pmatrix}.
\end{align*}
This shows that $J$ is constant, and hence $\theta$ moves on the torus at a constant velocity. Because the velocities are typically not co-related to each other, the dynamics are ``quasi-periodic." See, for example,  \citet{Scott_Dumas2014-gl} for more details.

As seen above, integrable Hamiltonian systems are quasi-periodic.
Note that general Hamiltonian systems are not necessarily quasi-periodic and neither are HNNs.
However, for the HNNs that are trained to model integrable systems, the quasi-periodic behaviors are preferably maintained. 
When the modeling error is sufficiently small, this is considered as a perturbation problem.
The perturbation theory of Hamiltonian systems has been investigated from various perspectives.
For example, perturbed integrable Hamiltonian systems are in general no longer integrable; hence, approximation of integrable Hamiltonian systems by integrable neural network models appears to be difficult.
Fortunately, however, the KAM theory shows that even though the perturbed system is not integrable, it maintains the quasi-periodic behaviors described above under certain conditions.

The KAM theorem has many variants under various conditions. The following variant~\citep{Scott_Dumas2014-gl} is typical: 
\begin{thm}[KAM Theorem]
	Let $\theta$ and $J$ be the action-angle variables for a $C^\infty$ integrable Hamiltonian $H_0: \R^{2M} \to \R$ with $M \geq 2$.
	If $H_0$ is non-degenerate, that is,
	\begin{align}\label{eq:non-degenerate}
		\det \frac{\partial^2 H_0}{\partial J^2} \neq 0,
	\end{align}
	for the perturbed system
	$
		H(\theta, J) = H_0(J) + \varepsilon F(\theta, J, \varepsilon)
	$
	by $F \in C^\infty$, there exists $\varepsilon_0$ such that if $\varepsilon F < \varepsilon_0$, there exists a set of $M$-dimensional tori that are invariant under the perturbed flow. For each invariant torus, the flow of the perturbed system $H$ is quasi-periodic. In addition, the set of invariant tori is large in the sense that its measure becomes full as $\varepsilon \to 0$.
\end{thm}

\begin{rem}
The last sentence -- {\it the set of invariant tori is large in the sense that its measure becomes full as $\varepsilon \to 0$} -- corresponds to the non-existence of so-called {\it resonance.}
If the perturbation added to the system is in resonance with the original system, the perturbation may grow rapidly and the behavior of the system may change significantly. This statement assures that for small perturbations, such a situation almost never occurs.
\end{rem}

\begin{rem}
It may be difficult to check whether the target system is integrable by using given data. One possibility is application of the Koopman operator, which makes
it possible to find the conserved quantities that the given data may admit. If a sufficient number of conserved quantities exist, it is highly likely that the target system is integrable.
\end{rem}

\section{Main Results} \label{sec:5}
\subsection{Universal Approximation Properties of HNNs}\label{sec:4.1}
For HNNs to be close to the true dynamics, a universal approximation theorem and a generalization error analysis are needed.
First, we show universal approximation theorems.

We first define some notation to describe the theorem.
$C^m(X)$ with the topology of the Sobolev space $W^{p,m}(X)$ is denoted by $S^m_p(X)$, where
$W^{p,m}(X)$ is a space of functions that admit weak derivatives up to the $m$th order that bounds $L^p$-norms.
Hence, $S^m_p(X)$ is the space of functions in $W^{p,m}(X)$ with (usual) derivatives;
for details on the Sobolev theory, see \citet{Adams2003-zt}.
$L^p$-norms of functions are denoted by $\| \cdot \|_{L^p}$, and those of vectors by $\| \cdot \|_p$.


\paragraph{Universal approximation theorem for HNNs}
The following theorem shows the universal approximation property of HNNs. 
\begin{thm}\label{thm:1}
	Let $H: \R^N \to \R$ be an energy function with the equation
	\begin{align*}
		\frac{\d u}{\d t} = S \frac{\partial H}{\partial u},
	\end{align*}
	where $u: t \in \R \mapsto u(t) \in \R^N$ and $S$ is a skew-symmetric $N \times N$ matrix. Suppose that the phase space $K$ of this system is compact and the right-hand side $S {\partial H}/{\partial u}$ is Lipschitz continuous. If the activation function $\sigma \neq 0$ belongs to $S^1_2(\R)$, then for any $\varepsilon > 0$, there exists a neural network $\hnn$ for which
	\begin{align*}
		\left\| S \frac{\partial H}{\partial u} - S \frac{\partial \hnn}{\partial u} \right\|_2 < \varepsilon
	\end{align*}
	holds. In addition, if the energy function is $C^\infty$, the function can be approximated by a $C^\infty$ neural network provided that the activation function is sufficiently smooth.
\end{thm}
An outline of the proof is as follows.
From the existence theorem of a solution to ordinary differential equations, the right-hand side of the equation must be Lipschitz continuous to guarantee the existence of a solution. Therefore, the differential of the $H$ is Lipschitz continuous,  which implies the smoothness of $H$. Then, the approximation property is obtained from the universal approximation theorem for smooth functions by \citet{Hornik1990-cg}. For the detailed proof, see the Appendix. 


\paragraph{HNNs with a coordinate transformation}

\begin{figure}[t]
	\centering
	\includegraphics[width=0.5\linewidth]{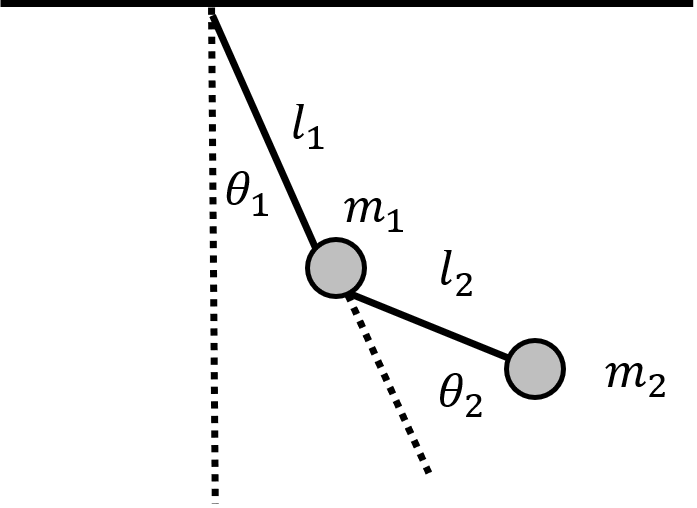}
	\caption{Double pendulum used as the target in the experiment for the illustration of the model with the coordinate transformations.}
	\label{fig:dp}
\end{figure}

The practical use of HNNs is hampered by the fact that the state variables must be represented by a specific coordinate, such as the generalized momentum;
however, the derivation of the generalized momentum requires the energy function, which is unknown. 
For example, the double pendulum in Figure \ref{fig:dp} exhibits the dynamics shown in Figure \ref{fig:double_pendulum_results}.
These are predicted by the models that are trained from the data of the state variables and their derivatives, not those of the generalized momenta.
HNNs failed to solve such problems because the data were not given in the canonical coordinate system. See the Appendix for details.
Based on this, we here propose a model with a coordinate transformation, such as the transformations that appear, for example, in \citet{Rana2020,Jin2020-qd}.

\begin{figure}
	\centering
	\subfigure[Ground truth]{\includegraphics[width=0.45\linewidth]{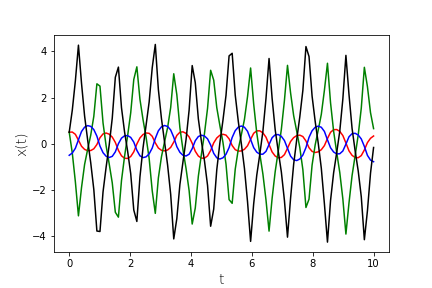}}
	\subfigure[Naive HNN model]{\includegraphics[width=0.45\linewidth]{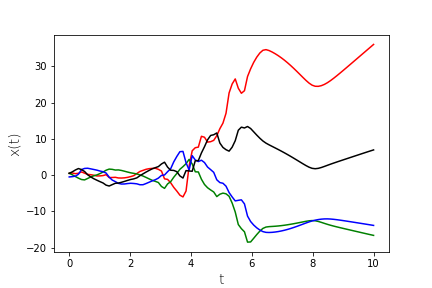}}\\
	\subfigure[HNN model with a coordinate transformation]{\includegraphics[width=0.45\linewidth]{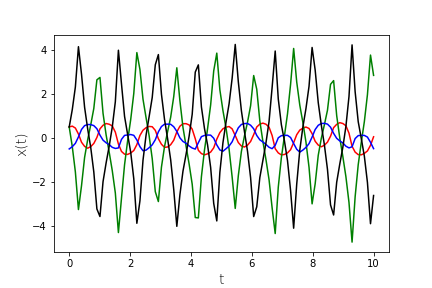}}
	\caption{Examples of the orbits predicted by a HNN and the model with coordinate transformations. Each component of $x(t) = [q_1(t), v_1(t), q_2(t), v_2(t)]$ is represented as red ($q_1$), green ($v_1$), blue ($q_2$), and black ($v_2$).}
	\label{fig:double_pendulum_results}
\end{figure}

Suppose that, although the given data point $x(t)$ is not represented by the canonical coordinate system, the data point $x(t)$ can be transformed into the canonical coordinate system by an unknown transformation $u(t)=u_{\mathrm{NN}}[x(t)]$.
By substituting $u=u_{\mathrm{NN}}(x)$ into the model equation $\eqref{eq:model}$, we obtain
\begin{align}\label{eq:transformed_model}
	\frac{\mathrm{d}x}{\mathrm{d}t}=(\frac{\partial u_{\NN}}{\partial x})^{-1} S (\frac{\partial u_\NN}{\partial x})^{-\top} \frac{\partial H_\NN}{\partial x}.
\end{align}

We show that model \eqref{eq:transformed_model} admits the same energetic property as the original equation and also the universal approximation property.  
\begin{thm}
	The model \eqref{eq:transformed_model} admits the energy conservation law in the sense that $\d \hnn/\d t = 0$.
\end{thm}

\begin{thm}\label{thm:2}
	Let $H: \R^N \to \R$ be an energy function for the equation
	\begin{align*}
		\frac{\mathrm{d}x}{\mathrm{d}t}=(\frac{\partial u}{\partial x})^{-1} S (\frac{\partial u}{\partial x})^{-\top} \frac{\partial H}{\partial x},
	\end{align*}
	where $x: t \in \R \mapsto x(t) \in \R^N$, $u: x \in \R^N \mapsto u(x) \in \R^N$, and $S$ is an $N \times N$ matrix.
	Suppose that the phase space $K$ of this system is compact and the right-hand side ${\partial H}/{\partial u}$ is Lipschitz continuous. Suppose also that $u$ is a $C^1$-diffeomorphism. If the functions $\sigma \neq 0$ and $\rho \neq 0$ belong to $S^1_2(\R)$, then for any $\varepsilon > 0$, there exist neural networks $\hnn$ with the activation functions $\sigma$ and $\unn$ with $\rho$ for which
	\begin{align*}
	&
		\left\|
		(\frac{\partial u}{\partial x})^{-1} S (\frac{\partial u}{\partial x})^{-\top} \frac{\partial H}{\partial x} \right.
		-
		\left. (\frac{\partial \unn}{\partial x})^{-1} S (\frac{\partial \unn}{\partial x})^{-\top} \frac{\partial \hnn}{\partial x}
		\right\|_2 
		\\ &
		< \varepsilon
	\end{align*}
	holds.
\end{thm}
For proofs, see Appendix. 



\subsection{Generalization Error Analysis of HNNs}
Next, we derive a generalization error bound for the standard HNN \eqref{eq:model} by employing a technique from statistical learning theory. More precisely, we adjust the technique so that an estimate on the energy gradient can be obtained.

\begin{rem}
Although the standard HNN without the coordinate transformations is considered below, the results can be extended to the general energy-based model with the coordinate transformations if the matrix $(\partial u/\partial x)^{-1}$ is bounded.
\end{rem}

In statistical learning theory, generalization error bounds are typically obtained by using the Rademacher complexities. See, for example, \citet{Bousquet2004-kl,Gine2016-zz, Shalev-Shwartz2014-ez, Steinwart2008-ba} for details.
\begin{dfn}
For a set $V \subset \R^n$,
\begin{align*}
\mathcal{R}_n(V) := \frac{1}{n} \mathbb{E}_{\sigma 
\sim \{-1, 1\}^n} \sup_{v \in V} \sum_{i=1}^n \sigma_{i} v_i
\end{align*}
 is called the Rademacher complexity of $V$. 
\end{dfn}

\begin{lem}\label{thm:Rad}
	Let $X$ and $Y$ be arbitrary spaces,
	$\mathcal{F} \subset \{ f:X \to Y \}$ be  a hypotheses class, and $L: Y \times Y \to [0, c]$ be a loss function.
	For a given data set $(x_i, y_i) \in X \times Y \ (i=1,\ldots,n)$, let $\mathcal{G}$ be defined by $\{ (x_i, y_i) \in X \times Y \mapsto L[y_i, h(x_i)] \in \R \mid h \in \mathcal{F}, i=1,\ldots,n\}$. Then, for any $\delta > 0$ and any probability measure $P$, we obtain, with a probability of at least $(1-\delta)$ with respect to the repeated sampling of $P^n$-distributed training data, the following:
	\begin{align*}
&		E[L(Y, h(X))] \leq \\
&		\frac{1}{n} \sum_{i=1}^n L(y_i, h(x_i))
		+ 2 \mathcal{R}_n(\mathcal{G}) + 3 c \sqrt{\frac{2 \ln \frac{4}{\delta}}{n}}
	\end{align*}
	for all $h \in \mathcal{F}$.
\end{lem}

The Rademacher complexity is known to be bounded by using the covering number.

\begin{dfn}
Let $V$ and $V^\prime$ be subsets of $\R^n$. 
$V$ is $r$-covered by $V^\prime$ with respect to the metric function defined by the norm $\| \cdot \|$ 
if for all $v \in V$, there exists a $v^\prime \in V^\prime$ such that
$\| v - v^\prime \| < r$.  The covering number $N(r, V, \| \cdot \|)$ of $V$ is the minimum number of elements of a set that $r$ covers $V$.
$N(r, V, \| \cdot \|)$ is also denoted by $N(r, V)$ if the metric is clear from the context.
\end{dfn}

\begin{lem}
If $\sqrt{\log N(c 2^{-k}, V)} \leq \alpha + k \beta$ for some $\alpha$ and $\beta$, then 
$
    \mathcal{R}_n(V) \leq {6 c(\alpha + 2 \beta)}/{n} .
   $ 
\end{lem}

Thus, if the covering number is estimated for a HNN, the bound on the generalization error is obtained.  
To this end, we suppose that the model is trained by minimizing the $p$-norm of the error in the right-hand side of the model. More precisely, for the hypothesis $h: u_j \mapsto  S \frac{\partial \hnn(u_j)}{\partial u}$, we consider the loss function
\begin{align}
L[\nabla H(u_j), h(u_j)] 
= 
\left\|\frac{\partial H(u_j)}{\partial u} - \frac{\partial \hnn(u_j)}{\partial u} \right\|^p_p,
\end{align}
where $u_i$ are training data. We denote the Lipschitz constant of the loss function associated with the above by $\rho_p$. Of course, $p=2$ is typically used; however, we show below that $p > 2M$ is useful to obtain an $L^\infty$ bound on the Hamiltonian.

\begin{rem}
The training can be performed also by using the symplectic gradient: 
\begin{align}
\left\|S \frac{\partial H(u_j)}{\partial u} - S \frac{\partial \hnn(u_j)}{\partial u} \right\|^p_p.
\end{align}
In that case, the results will be slightly modified using the norm of $S$; however, we omit this for simplicity.
\end{rem}

A bound of the covering number is derived as follows.

\begin{thm}\label{thm:gea}
Suppose that the hypotheses class $\mathcal{F}$ consists of 
 multi-layer perceptrons $\fnn$ that have $\rho_{\sigma_j}$-Lispchitz activation functions $\sigma_j (j=1,\ldots,n_l)$, for which the derivatives are $\rho_j^\prime$-Lipschitz continuous and bounded by $\sup |\sigma_j^\prime| < c_{\sigma_j}$. Suppose also that the matrices $A_j^\top (j=1,\ldots,n_{l}+1)$ in the linear layers in the perceptrons have the bounded norm $|A_j^\top| < c_{A_j}$:
 \begin{align*}
& \mathcal{F} = \{ \fnn(u) \mid 
\\
& A_{n_l+1} \sigma_{n_l}(A_{n_l} \sigma_{n_{l-1}}[\cdots \sigma_1(A_1 u + b_1)\cdots ] + b_{n_l}  ) + b_{n_l+1} \},
 \end{align*}
where $b_j$'s are vectors.
Let $\mathcal{G}$ be defined by $\{ L[\nabla H(u_i), h(u_i)] \mid h \in \mathcal{F}\}$ with the $\rho_p$-Lipschitz continuous loss function $L$. 
In addition, suppose that the phase space is compact so that the data $u_i (i=1,\ldots,n)$ are in a bounded set with the bound $\| u_i \| < c_u$.
Then, the covering number of $\mathcal{G}$ is estimated by
\begin{multline*}
 N(\varepsilon, \mathcal{G}) 
\leq \\  
\Bigl( \frac{2 \rho_p c_u c_{A_{n_{l+1}}} \rho^\prime_{\sigma_{n_l}} (\prod_{j=1}^{n_{l}-1} \rho_{\sigma_j}) (\prod_{j=1}^{n_l-1} c_{\sigma_j}) (\prod_{j=1}^{n_l} c_{A_j})^2}{\varepsilon}
\\ 
+ 1\Bigr)^n.
\end{multline*}
\end{thm}

To prove this theorem, we use the following lemmas, which are typically used to estimate the covering numbers~\cite{Shalev-Shwartz2014-ez}.
\begin{lem}
Let $B$ be a unit ball in $\R^n$. Then,
$N(\varepsilon, B, \| \cdot \|_2) \leq \left( \frac{2}{\varepsilon} + 1\right)^n$.
\end{lem}
\begin{lem}
 Suppose that functions $\phi_i: \R \to \R, \ i = 1,2,\ldots,n$ are $\rho$-Lipschitz continuous. Then, for $V \subset V^n$, 
$
                  N(\varepsilon, \vec{\phi} \circ V) \leq  N(\frac{\varepsilon}{\rho}, V),
                  $
                    where for $v \in \R^n$, 
$
        \vec{\phi}(v) := [\phi_1(v_1),\ldots,\phi_n(v_n)], \
        \vec{\phi} \circ V := \{ \vec{\phi}(v) \mid v \in V\}.
        $
\end{lem}

\begin{proof}[Proof of Theorem \ref{thm:gea}]
    To simplify the discussion, we will estimate the covering number of the following perceptron
    \begin{align*}
        \fnn(u) = A_{3} \sigma_{2}[A_{2} \sigma_1(A_1 u + b_1) + b_{2}] + b_{3}.     \end{align*}
    Because the proof for general cases is exactly the same, we need to estimate the covering number of the gradient of $\fnn$, which is written as
    \begin{align*}
        \nabla \fnn(u) = A_1^\top (D \sigma_1) A_2^\top (D \sigma_2) A_3^\top,
    \end{align*}
    where $D \sigma_2$ and $D \sigma_1$ are Jacobian matrices. 
    These Jacobian matrices are evaluated at $u = A_{2} \sigma_1(A_1 u + b_1) + b_{2}$ and $A_1 u + b_1$, respectively. 
    We first estimate the covering number associated with $D \sigma_2$. 
    $D \sigma_2$ has the same architecture as a multi-layer perceptron, except that the last activation function is replaced by the differential $\sigma_2^\prime$ of $\sigma_2$:
    \begin{align*}
    D \sigma_2 = \sigma_2^\prime [A_2 \sigma_1(A_1 u + b_1) + b_{2}].
    \end{align*}
    Assuming that the input data are in the ball $B_{c_u}$ with radius $c_u$, we obtain
    \begin{align*}
    N(\varepsilon, B_{c_u}) \leq \left( \frac{2 c_u}{\varepsilon} + 1\right)^n.
    \end{align*}
    Then, because the norms of $A_1$, $A_2$ are bounded by $c_{A_1}, c_{A_2}$, matrix multiplications by these matrices are $c_{A_1}$- and $c_{A_2}$-Lipschitz continuous, respectively. In addition, $\sigma_1$ is $\rho_1$-Lipschitz and $\sigma_2^\prime$ is $\rho_2^\prime$-Lipschitz continuous. Therefore, the covering number associated with $D \sigma_2$ is estimated by
    \begin{align*}
    N(\varepsilon, D \sigma_2) \leq \left( \frac{2 \rho_{1} \rho^\prime_2 c_{A_1} c_{A_2} c_u}{\varepsilon} + 1\right)^n.
    \end{align*}
    Finally, because $\sigma_1^\prime$ is assumed to be bounded by $c_{\sigma_1}$, the norms of the matrices other than $D \sigma_2$ in $\nabla \fnn$ are bounded as follows:
    $
 \| A_1^\top \| < c_{A_1}, \        
 \| A_2^\top \| < c_{A_2}, \        
 \| A_3^\top \| < c_{A_3}, \        
 \| D \sigma_1 \| < c_{\sigma_1}.     
 $
    Because the loss function is assumed to be $\rho_p$-Lipschitz continuous, we obtain the estimate
     \begin{align*}
    N(\varepsilon, \mathcal{G}) \leq \left( \frac{2 \rho_p \rho_{1} \rho^\prime_2 c_{\sigma_1} (c_{A_1} c_{A_2})^2 c_{A_3} c_u}{\varepsilon} + 1\right)^n.
    \end{align*}   
\end{proof}

\subsection{$L^\infty$ estimate on the error in the Hamiltonian}
The generalization error analysis in Theorem \ref{thm:gea} 
 shows that, at a certain probability, the expectation of the loss function can be bounded.
If this bound certainty holds and if the training is performed by minimizing the $p$-norm with $p > 2M$, we can derive an $L^\infty$ estimate on the Hamiltonian for the standard HNN \eqref{eq:model} by applying the Poincar\'e inequality and the Sobolev inequality under Assumption 1.

\noindent \textbf{Assumption 1}
There exists a density $f_P$ for measure $P$ with $\inf f_P > 0$.


\begin{rem}
The condition $p > 2M$ is not required in practice because of the well-known equivalence of the norms in finite dimensional spaces; for example, if the standard 2-norm is small enough, then the $p$-norm is also small.
However, when the dimension $2M$ is large, the 2-norm needs to be very small to bound the $p$-norm because the constant in the inequality used to bound the $p$-norm depends on the dimension. Therefore, it is preferable to minimize the $p$-norm in such cases.
\end{rem}

\begin{thm}[Poincar\'e inequality]
	Suppose that $1 \leq p \leq \infty$ and $\Omega \subset \R^{2 M}$ is bounded. Then there exists a constant $c_\mathrm{p}$ such that, for any $H \in S_p^1(\Omega)$,
	\begin{align*}
		\int_\Omega |H(u)- \bar{H} |^p \d u \leq c_\mathrm{p} \left \| \frac{\partial H}{\partial u} \right\|_p^p,\
		\bar{H} =\frac{1}{\int_\Omega \d u}\int_\Omega H(u) \d u.
	\end{align*}
	The constant $c_\mathrm{p}$ is called the Poincar\'e constant.
\end{thm}

\begin{thm}[Sobolev inequalities]
There exist constants $c_1, c_2$ such that, if $l p > 2M$,
\begin{align*}
	 & \| e \|_{L^\infty}(\R^{2M}) \leq c \| e \|_{W^{p,l}}(\R^{2M}), 
	 \\ &
	\| e \|_{L^\infty}(\mathbb{T}^{2M}) \leq c \| e \|_{W^{p,l}}(\mathbb{T}^{2M}).
\end{align*}
\end{thm}

By using these inequalities along with the invariance of the Hamilton equation under the constant shift of the energy function, we obtain an error bound on the Hamiltonian.

\begin{lem}
	Among the energy functions that yield the target Hamilton equation, we choose the one for which
	\begin{align}\label{eq:meanvalue}
		\int H(u) \d u = \int \hnn(u) \d u
	\end{align}
	holds, so that the error function has zero mean:
	$
		e(u) :=H(u) - \hnn(u), \
		\bar{e}(u) := 0.
	$
	Then, 
	\begin{align*}
		\int_\Omega |e(u)|^p \d u \leq c_\mathrm{p} \left \| \frac{\partial e}{\partial u} \right\|_{L^p}^p.
	\end{align*}
\end{lem}

From the above estimate, we obtain
\begin{align*}
	& \int \left\| \frac{\partial H(u)}{\partial u} -  \frac{\partial \hnn(u)}{\partial u} \right\|^p_p \d P
	\\ &
	\leq \frac{1}{n} \sum_{i=1}^n L[Y_i, h(X_i)]
		+ 2 \mathcal{R}_n(\mathcal{G}) + 3 c \sqrt{\frac{2 \ln \frac{4}{\delta}}{n}}.
\end{align*}
By using the density $f_P$ for the measure $P$, we obtain
\begin{align*}
	& \inf f_P \int \left\| \frac{\partial H(u)}{\partial u} -  \frac{\partial \hnn(u)}{\partial u} \right\|^p_p \d u
	\\ 
	&
	\leq
	\int \left\| \frac{\partial H(u)}{\partial u} -  \frac{\partial \hnn(u)}{\partial u} \right\|^p_p \d P,
\end{align*}
which gives us
\begin{align*}
	&\int \left\| \frac{\partial H(u)}{\partial u} -  \frac{\partial \hnn(u)}{\partial u} \right\|^p_p \d u
	\\
	 & \leq
	\frac{1}{\inf f_P} \int \left\| \frac{\partial H(u)}{\partial u} -  \frac{\partial \hnn(u)}{\partial u} \right\|^p_p \d P
	\\
	 &
	\leq
	\frac{1}{\inf f_P} \left(
	\frac{1}{n} \sum_{i=1}^n L[Y_i, h(X_i)]
		+ 2 \mathcal{R}_n(\mathcal{G}) + 3 c \sqrt{\frac{2 \ln \frac{4}{\delta}}{n}}
	\right).
\end{align*}

We note that the left-hand side is the Sobolev norm of the error in $W^{p,l}$; then,  
under the assumption that $p > 2 M$, we can use the Sobolev inequality 
 to obtain
\begin{align*}
& (\sup_u \|  H(u) - \hnn(u) \|)^p 
\leq c^p
\left\| \frac{\partial H(u)}{\partial u} -  \frac{\partial \hnn(u)}{\partial u} \right\|^p_p
\\
& \leq
	\frac{c^p}{\inf f_P} \left(
	\frac{1}{n} \sum_{i=1}^n L[Y_i, h(X_i)]
		+ 2 \mathcal{R}_n(\mathcal{G}) + 3 c \sqrt{\frac{2 \ln \frac{4}{\delta}}{n}}
	\right),
\end{align*}
which ensure that $H$ and $\hnn$ are close in terms of the function values.

\subsection{KAM Theory for HNNs}
The universal approximation property shown in the previous sections guarantees that the value of MSE  can be made arbitrarily small by training; however, in actual training, a finite error remains. 
In this section, as an application of the generalization bound, we apply the KAM theory to theoretically investigate the trained standard HNN model \eqref{eq:model} in such cases by assuming that the target system is integrable. 

We make a few assumptions that are needed for the application of the KAM theory. 

\noindent
\textbf{Assumption 2}\
The dimension of the phase space is assumed to be $2 M$ with $M \geq 2$.

\noindent
\textbf{Assumption 3}\
The target system is an integrable Hamiltonian system with the conserved quantities $F_1,\ldots,F_M$. The series $c_1,\ldots,c_M$ exists such that $K = \cap_{i=1}^M F_i^{-1}(c_i)$ is connected and compact. 

Under the above assumptions, from the Liouville--Arnold theorem there exist a neighborhood $\mathcal{N}$ of $K$, $\mathcal{U} \subset \R^n$ and a coordinate transform
	\begin{align}
		\phi: (\theta, J) \in \mathbb{T}^n \times \mathcal{U} \to \phi(\theta, J) \in \mathcal{N},
	\end{align}
		such that
	the transformed system is the Hamilton equation. 
	Following the usual setting of the KAM theorem, we consider the target system and the Hamiltonian equation in the transformed coordinate $\mathbb{T}^n \times \mathcal{U}$. 

\noindent
\textbf{Assumption 4}\
The Hamiltonian $H: \mathbb{T}^n \times \mathcal{U} \to \R$ of the target system is $C^\infty$ and non-degenerate. 
The activation functions of the HNNs used are in $C^\infty$.

\noindent
\textbf{Assumption 5}\
From the generalization error analysis in the previous section, 
we have essentially shown that 
if $p > 2 M$, with at least probability $1-\delta$, it holds that
\begin{align*}
    \sup |H(u) - \hnn(u)| < c_1 L_{\mathrm{train}} + c_2 R_n + c_3 \sqrt{\frac{\ln \frac{1}{\delta}}{n}}
\end{align*}
with constants $c_1, c_2$, and $c_3$, where $R_n$ is a bound on the Rademacher complexity. We assume that the training was performed with $p> 2M$ and the above statement certainly holds.


Using these assumptions, we obtain Theorem 10. 
\begin{thm}\label{thm:kamforhnn}
	Let the threshold of the KAM theorem be $\varepsilon_0$ and $\delta$ be
	\begin{align*}
	    \delta = \exp\left(- n \left(\frac{\varepsilon_0 - c_1 L_{\mathrm{train}} - c_2 R_n}{c_3}\right)^2 \right).
	\end{align*}
	Under the above assumptions,
	with a probability of at least $(1-\delta)$, a set of invariant tori exists for the trained model $\hnn$.
\end{thm}

\begin{proof}
    It is confirmed by a straightforward calculation that if $\delta$ is given as described above, it holds that 
    $\sup |H(u) - \hnn(u)| < \varepsilon_0$, and hence the assumption of the KAM theorem is satisfied. 
\end{proof}

\begin{rem}
    As mentioned in Remark 1, the KAM theorem also shows that the invariant tori become larger when the perturbation becomes smaller. Hence, if the generalization error is small enough, the size of the tori is expected to be large. 
\end{rem}

Note that general Hamiltonian systems, and hence general HNNs, are not quasi-periodic. Therefore, a model that approximates a quasi-periodic Hamilton equation may be (in some sense) \textit{approximately} quasi-periodic, but it is not necessarily \textit{strictly} quasi-periodic. This theorem states that \textit{the trained model can be strictly quasi-periodic even if the training loss does not completely vanish.}


\paragraph{Numerical Example: Learning the Zabusky and Kruskal Experiment}

As a numerical experiment, we trained a HNN\footnote{We use the HNN code for the KdV equation provided by \url{https://github.com/tksmatsubara/discrete-autograd} (MIT License).} so that the dynamics of the KdV equation is learned by using the data from the experiment by \citet{Zabusky1965-mn}, in which a nontrivial recurrence of initial states is reported.


The KdV equation is derived from the following energy function $H$:
\begin{align*}
	H(u)=\int \left[\frac{1}{6}\alpha u^3-\frac{1}{2}\beta \left(\frac{\partial u}{\partial x}\right)^2 \right] \d x.
\end{align*}
In fact, under the periodic boundary condition, the variational derivative is
\begin{align*}
	\frac{\delta H}{\delta u} = \int \left[\frac{1}{2}\alpha u^2+\beta \frac{\partial^2 u}{\partial x^2} \right] \d x.
\end{align*}
Then, the time evolution is expressed as a Hamiltonian equation:
\begin{align*}
	\frac{\partial u}{\partial t} = \frac{\partial}{\partial x}\frac{\delta H}{\delta u}= \alpha u\frac{\partial u}{\partial x}+\beta \frac{\partial^3 u}{\partial x^3}.
\end{align*}

For spatial discretization, we used the forward and backward difference operators,
\begin{align*}\arraycolsep=1.1mm
	& D_\mathrm{f}
	:= \frac{1}{\Delta x}\begin{pmatrix}
		-1     & 1      & 0      & \cdots & 0      & 0      & 0      \\
		0      & -1     & 1      & \cdots & 0      & 0      & 0      \\
		0      & 0      & -1     & \cdots & 0      & 0      & 0      \\
		\vdots & \ddots & \ddots & \cdots & \ddots & \ddots & \vdots \\
		0      & 0      & 0      & \cdots & -1     & 1      & 0      \\
		0      & 0      & 0      & \cdots & 0      & -1     & 1      \\
		1      & 0      & 0      & \cdots & 0      & 0      & -1
	\end{pmatrix} 
\mbox{\ and} \\
	& D_\mathrm{b}
	:= \frac{1}{\Delta x}\begin{pmatrix}
		1      & 0      & 0      & \cdots & 0      & 0      & -1     \\
		-1     & 1      & 0      & \cdots & 0      & 0      & 0      \\
		0      & -1     & 1      & \cdots & 0      & 0      & 0      \\
		\vdots & \ddots & \ddots & \cdots & \ddots & \ddots & \vdots \\
		0      & 0      & 0      & \cdots & 1      & 0      & 0      \\
		0      & 0      & 0      & \cdots & -1     & 1      & 0      \\
		0      & 0      & 0      & \cdots & 0      & -1     & 1
	\end{pmatrix},
\end{align*}
respectively.
The central difference operator $D$ is their mean, specifically $D=\frac{1}{2}(D_\mathrm{f}+D_\mathrm{b})$ and that for the second derivative is $D_2 = D_\mathrm{f} D_\mathrm{b} =  D_\mathrm{b} D_\mathrm{f}$. 
Using these difference operators, the equation is semi-discretized as
\begin{align*}
	& H(u) =\sum_x \left[\frac{1}{6}\alpha u^3-\frac{1}{2}\beta \frac{(D_\mathrm{f} u)^2+(D_\mathrm{b} u)^2}{2} \right] \Delta x,\\
	& \frac{\d u}{\d t} = D \frac{\partial H}{\partial u}=D\left(\frac{1}{2}\alpha u^2+\beta D u\right).
\end{align*}

%
Following 
\citet{Zabusky1965-mn}, we set the parameters to $\alpha=-1.0$ and $\beta=-0.022^2$, set the width of phase space to 2.0, and used the initial condition $u(0,x)$ to $u(0,x)=\cos (x\pi)$.
We discretized the system with the spatial and temporal mesh sizes of $\Delta x=0.1$ and $\Delta t=0.01$.
We obtained an orbit for 200 time steps from the initial condition using the fifth-order Dormand--Prince method with the absolute and relative tolerances of $10^{-10}$ and $10^{-8}$.

We performed the experiments on an NVIDIA TITAN V with double precision.
We employed a three-layered convolutional neural network with kernel sizes of 3, 1, and 1.
The number of hidden channels was 200, the number of output channels was 1, the activation function was the $\tanh$ function, and each weight parameter was initialized as a random orthogonal matrix.
We summed up the output in the spatial direction and obtained the global energy.
We used the whole orbit at every iteration, and minimized the mean squared error of the time derivative 
as the loss function using the Adam optimizer with a learning rate of $10^{-3}$ for 10,000 iterations; the error reached a maximum of $1.37\times 10^{-3}$.
Given the true dynamics $u$, the absolute error between the energy function $H$ and the neural network $\hnn$ was $1.31\times 10^{-4}$ on average and $2.51\times 10^{-4}$ at most.

Using the true model and the trained neural network, we also obtained orbits for 1100 time steps from the same initial condition, as shown in the second and third panels of Fig.~\ref{fig:kdvlong}, respectively.
In the top panels, blue and orange lines denote the true state $u$ and the state predicted by the trained neural network $\unn$ at $t=$0.0, 2.0, 4.8, and 9.8.
The bottom panel shows the energy function $H$ given the predicted states $u$ and $\unn$.
Due to the non-zero training error, more waves incur a larger error.
Nonetheless, at around $t=9.8$, the true model and learned neural network reproduce $\sin$ waves, which are given as the initial condition, and the energy error is restored to zero; they exhibit quasi-periodic behaviors.

\begin{figure*}[t]
	\includegraphics[width=1.0\textwidth]{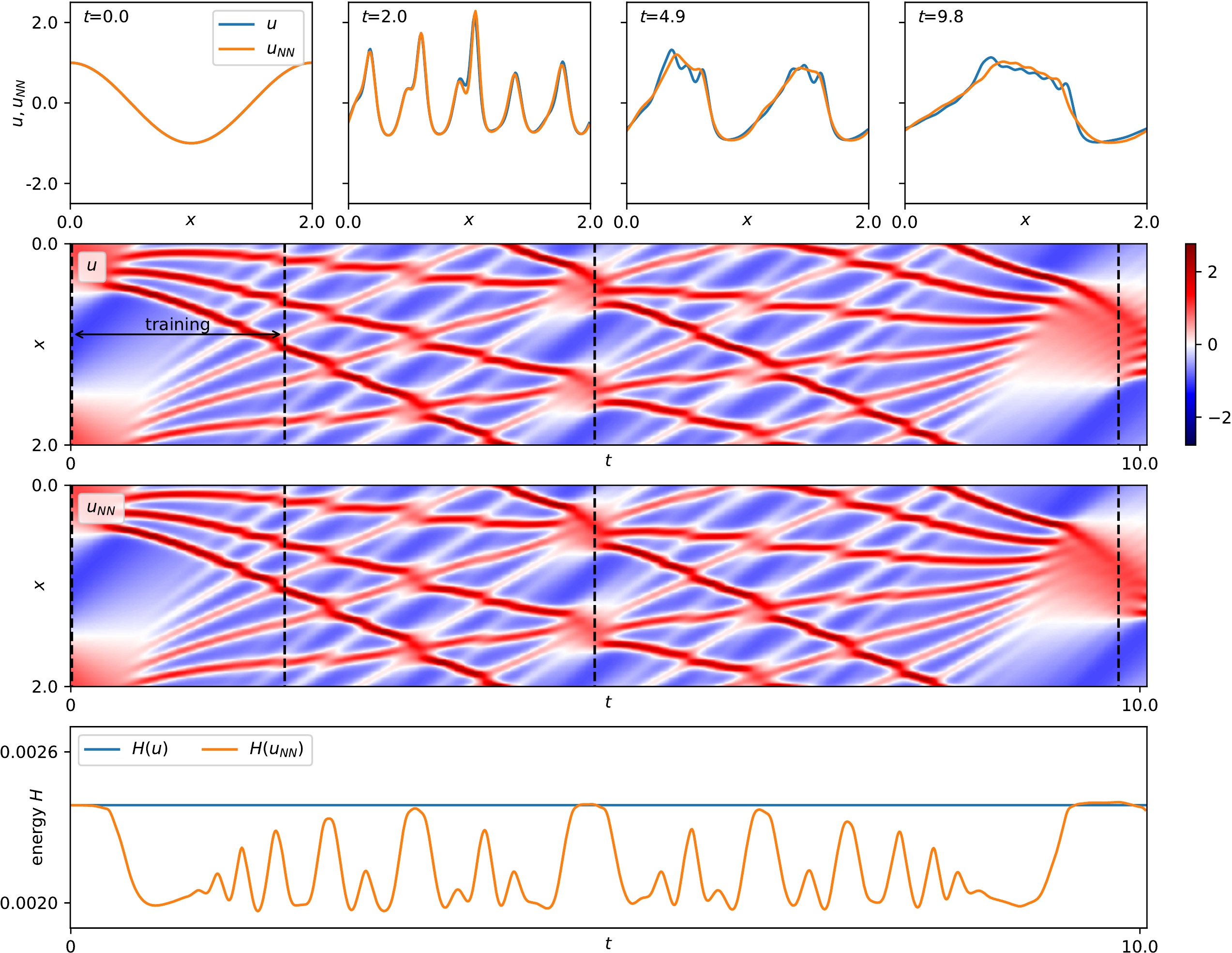}
	\caption{Results of training in the Zabusky and Kruskal experiment~\cite{Zabusky1965-mn}.
		(top panels) The predicted states at $t=$0.0, 2.0, 4.8, and 9.8.
		(second panel) The true dynamics $u$.
		(third panel) The dynamics $\unn$ modeled by a neural network.
		(bottom panel) The energy function $H$ given the true dynamics $u$ and modeled dynamics $\unn$.
	}\label{fig:kdvlong}
\end{figure*}

\section{Concluding Remarks}
We analyzed the behavior of HNNs with non-zero learning errors by combining the KAM theory and statistical machine learning. We investigated the approximation properties of deep energy-based models, including HNNs. More precisely, we proved the persistence of the quasi-periodic behaviors of integrable Hamiltonian systems with a high probability even when the loss function is not perfectly zero.
Further, we provided a generalization error bound and universal approximation theorems for HNNs 
to ensure that the loss function can be sufficiently small for application of the KAM theorem. 
Meanwhile, in the recent research on this type of model, numerically integrated gradients are often used for training. Similar results should be obtained for such cases; however, rigorous discussion is needed.

\bibliography{example_paper}

\begin{thebibliography}{39}
\providecommand{\natexlab}[1]{#1}

\bibitem[{Adams and Fournier(2003)}]{Adams2003-zt}
Adams, R.~A.; and Fournier, J. J.~F. 2003.
\newblock \emph{Sobolev Spaces}.
\newblock Elsevier.

\bibitem[{Bharadwaj, Li, and Demanet(2020)}]{Bharadwaj2020-fr}
Bharadwaj, P.; Li, M.; and Demanet, L. 2020.
\newblock {SymAE}: An autoencoder with embedded physical symmetries for passive
  time-lapse monitoring.
\newblock In \emph{{SEG} Technical Program Expanded Abstracts 2020}. Society of
  Exploration Geophysicists.

\bibitem[{Bondesan and Lamacraft(2019)}]{Bondesan2019-vm}
Bondesan, R.; and Lamacraft, A. 2019.
\newblock {Learning Symmetries of Classical Integrable Systems}.
\newblock In \emph{ICML 2019 Workshop on Theoretical Physics for Deep
  Learning}.

\bibitem[{Bousquet, Boucheron, and Lugosi(2004)}]{Bousquet2004-kl}
Bousquet, O.; Boucheron, S.; and Lugosi, G. 2004.
\newblock Introduction to Statistical Learning Theory.
\newblock In Bousquet, O.; von Luxburg, U.; and R{\"a}tsch, G., eds.,
  \emph{Advanced Lectures on Machine Learning}, 169--207. Berlin, Heidelberg:
  Springer Berlin Heidelberg.

\bibitem[{Caginalp(1986)}]{Caginalp1986-ur}
Caginalp, G. 1986.
\newblock An analysis of a phase field model of a free boundary.
\newblock \emph{Archive for Rational Mechanics and Analysis}, 92(3): 205--245.

\bibitem[{Chen et~al.(2018)Chen, Rubanova, Bettencourt, Duvenaud, Chen,
  Rubanova, Bettencourt, and Duvenaud}]{Chen2018-yw}
Chen, T.~Q.; Rubanova, Y.; Bettencourt, J.; Duvenaud, D.; Chen, R. T.~Q.;
  Rubanova, Y.; Bettencourt, J.; and Duvenaud, D. 2018.
\newblock {Neural Ordinary Differential Equations}.
\newblock In \emph{Advances in Neural Information Processing Systems
  (NeurIPS)}.

\bibitem[{Chen et~al.(2020)Chen, Zhang, Arjovsky, and Bottou}]{Chen2019-eo}
Chen, Z.; Zhang, J.; Arjovsky, M.; and Bottou, L. 2020.
\newblock {Symplectic Recurrent Neural Networks}.
\newblock In \emph{International Conference on Learning Representations
  (ICLR)}.

\bibitem[{Cranmer et~al.(2020)Cranmer, Greydanus, Hoyer, Battaglia, Spergel,
  and Ho}]{Cranmer2020-zw}
Cranmer, M.; Greydanus, S.; Hoyer, S.; Battaglia, P.; Spergel, D.; and Ho, S.
  2020.
\newblock {Lagrangian Neural Networks}.
\newblock \emph{ICLR 2020 Deep Differential Equations Workshop}.

\bibitem[{Desai and Roberts(2020)}]{Desai2020-ar}
Desai, S.; and Roberts, S. 2020.
\newblock {VIGN: Variational Integrator Graph Networks}.
\newblock \emph{arXiv}.

\bibitem[{DiPietro, Xiong, and Zhu(2020)}]{DiPietro2020-mj}
DiPietro, D.~M.; Xiong, S.; and Zhu, B. 2020.
\newblock {Sparse Symplectically Integrated Neural Networks}.
\newblock In \emph{Advances in Neural Information Processing Systems
  (NeurIPS)}.

\bibitem[{F{\'e}joz(2013)}]{Fejoz2013-ud}
F{\'e}joz, J. 2013.
\newblock On ``{A}rnold's theorem'' on the stability of the solar system.
\newblock \emph{Discrete \& Continuous Dynamical Systems}, 33(8): 3555--3565.

\bibitem[{Feng et~al.(2020)Feng, Wang, Yang, and Wang}]{Feng2020-hd}
Feng, Y.; Wang, H.; Yang, H.; and Wang, F. 2020.
\newblock {Time-Continuous Energy-Conservation Neural Network for Structural
  Dynamics Analysis}.
\newblock \emph{arXiv}.

\bibitem[{Finzi et~al.(2020)Finzi, Stanton, Izmailov, and
  Wilson}]{Finzi2020-zh}
Finzi, M.; Stanton, S.; Izmailov, P.; and Wilson, A.~G. 2020.
\newblock {Generalizing Convolutional Neural Networks for Equivariance to Lie
  Groups on Arbitrary Continuous Data}.
\newblock In \emph{International Conference on Machine Learning (ICML)},
  3165--3176.

\bibitem[{Finzi, Wang, and Wilson(2020)}]{Finzi2020-qa}
Finzi, M.; Wang, K.~A.; and Wilson, A.~G. 2020.
\newblock {Simplifying Hamiltonian and Lagrangian Neural Networks via Explicit
  Constraints}.
\newblock In \emph{Advances in Neural Information Processing Systems
  (NeurIPS)}.

\bibitem[{Furihata and Matsuo(2010)}]{Furihata2010}
Furihata, D.; and Matsuo, T. 2010.
\newblock \emph{{Discrete Variational Derivative Method: A Structure-Preserving
  Numerical Method for Partial Differential Equations}}.
\newblock Chapman and Hall/CRC.

\bibitem[{Galioto and Gorodetsky(2020)}]{Galioto2020-xc}
Galioto, N.; and Gorodetsky, A.~A. 2020.
\newblock {Bayesian Identification of Hamiltonian Dynamics from Symplectic
  Data}.
\newblock In \emph{{IEEE} Conference on Decision and Control ({CDC})},
  1190--1195.

\bibitem[{Gin{\'e} and Nickl(2016)}]{Gine2016-zz}
Gin{\'e}, E.; and Nickl, R. 2016.
\newblock \emph{Mathematical Foundations of {Infinite-Dimensional} Statistical
  Models}.
\newblock Cambridge University Press.

\bibitem[{Greydanus, Dzamba, and Yosinski(2019)}]{Greydanus2019-gy}
Greydanus, S.; Dzamba, M.; and Yosinski, J. 2019.
\newblock {Hamiltonian Neural Networks}.
\newblock In \emph{Advances in Neural Information Processing Systems
  (NeurIPS)}.

\bibitem[{Hornik, Stinchcombe, and White(1990)}]{Hornik1990-cg}
Hornik, K.; Stinchcombe, M.; and White, H. 1990.
\newblock {Universal approximation of an unknown mapping and its derivatives
  using multilayer feedforward networks}.
\newblock \emph{Neural Networks}, 3(5): 551--560.

\bibitem[{Jin et~al.(2020{\natexlab{a}})Jin, Zhang, Kevrekidis, and
  Karniadakis}]{Jin2020-qd}
Jin, P.; Zhang, Z.; Kevrekidis, I.~G.; and Karniadakis, G.~E.
  2020{\natexlab{a}}.
\newblock {Learning Poisson systems and trajectories of autonomous systems via
  Poisson neural networks}.
\newblock \emph{arXiv}.

\bibitem[{Jin et~al.(2020{\natexlab{b}})Jin, Zhang, Zhu, Tang, and
  Karniadakis}]{Jin2020-zs}
Jin, P.; Zhang, Z.; Zhu, A.; Tang, Y.; and Karniadakis, G.~E.
  2020{\natexlab{b}}.
\newblock {SympNets: Intrinsic structure-preserving symplectic networks for
  identifying Hamiltonian systems}.
\newblock \emph{Neural Networks}, 132: 166--179.

\bibitem[{Jin, Lin, and Li(2020)}]{Jin2020-wl}
Jin, Z.; Lin, J. Y.-Y.; and Li, S.-F. 2020.
\newblock Learning Principle of Least Action with Reinforcement Learning.
\newblock \emph{arXiv}.

\bibitem[{Laskar(1996)}]{Laskar1996}
Laskar, J. 1996.
\newblock Large scale chaos and marginal stability in the solar system.
\newblock volume~64, 115--162.
\newblock Chaos in gravitational $N$-body systems (La Plata, 1995).

\bibitem[{Li et~al.(2020)Li, Dong, Zhang, and Wang}]{Li2020-rx}
Li, S.-H.; Dong, C.-X.; Zhang, L.; and Wang, L. 2020.
\newblock {Neural Canonical Transformation with Symplectic Flows}.
\newblock \emph{Physical Review X}, 10(2): 021020.

\bibitem[{Matsubara, Ishikawa, and Yaguchi(2020)}]{Matsubara2019-ey}
Matsubara, T.; Ishikawa, A.; and Yaguchi, T. 2020.
\newblock {Deep Energy-Based Modeling of Discrete-Time Physics}.
\newblock In \emph{Advances in Neural Information Processing Systems
  (NeurIPS)}.

\bibitem[{Miehe, Hofacker, and Welschinger(2010)}]{Miehe2010-jh}
Miehe, C.; Hofacker, M.; and Welschinger, F. 2010.
\newblock A phase field model for rate-independent crack propagation: Robust
  algorithmic implementation based on operator splits.
\newblock \emph{Computer Methods in Applied Mechanics and Engineering},
  199(45): 2765--2778.

\bibitem[{Rana et~al.(2020)Rana, Li, Fox, Boots, Ramos, and Ratliff}]{Rana2020}
Rana, M.~A.; Li, A.; Fox, D.; Boots, B.; Ramos, F.; and Ratliff, N. 2020.
\newblock {Euclideanizing Flows: Diffeomorphic Reduction for Learning Stable
  Dynamical Systems}.
\newblock In \emph{Conference on Learning for Dynamics and Control (L4DC)},
  volume 120.

\bibitem[{S{\ae}mundsson et~al.(2019)S{\ae}mundsson, Hofmann, Terenin, and
  Deisenroth}]{noauthor_undated-ko}
S{\ae}mundsson, S.; Hofmann, K.; Terenin, A.; and Deisenroth, M.~P. 2019.
\newblock {Variational integrator networks for physically meaningful
  embeddings}.
\newblock In \emph{Artificial Intelligence and Statistics (AISTATS)}, volume
  108, 3078--3087.

\bibitem[{Sanchez-Gonzalez et~al.(2019)Sanchez-Gonzalez, Bapst, Cranmer, and
  Battaglia}]{Sanchez-Gonzalez2019-js}
Sanchez-Gonzalez, A.; Bapst, V.; Cranmer, K.; and Battaglia, P. 2019.
\newblock {Hamiltonian Graph Networks with {ODE} Integrators}.
\newblock \emph{arXiv}.

\bibitem[{Scott~Dumas(2014)}]{Scott_Dumas2014-gl}
Scott~Dumas, H. 2014.
\newblock \emph{KAM Story, The: A Friendly Introduction To The Content,
  History, And Significance Of Classical Kolmogorov-Arnold-Moser Theory}.
\newblock World Scientific Publishing Company.

\bibitem[{Shalev-Shwartz and Ben-David(2014)}]{Shalev-Shwartz2014-ez}
Shalev-Shwartz, S.; and Ben-David, S. 2014.
\newblock \emph{Understanding Machine Learning: From Theory to Algorithms}.
\newblock Cambridge University Press.

\bibitem[{Steinbach(2009)}]{Steinbach2009-uy}
Steinbach, I. 2009.
\newblock Phase-field models in materials science.
\newblock \emph{Modelling and Simulation in Materials Science and Engineering},
  17(7): 073001.

\bibitem[{Steinwart and Christmann(2008)}]{Steinwart2008-ba}
Steinwart, I.; and Christmann, A. 2008.
\newblock \emph{Support Vector Machines}.
\newblock Springer Science \& Business Media.

\bibitem[{Teshima et~al.(2020)Teshima, Ishikawa, Tojo, Oono, Ikeda, and
  Sugiyama}]{teshima2020couplingbased}
Teshima, T.; Ishikawa, I.; Tojo, K.; Oono, K.; Ikeda, M.; and Sugiyama, M.
  2020.
\newblock Coupling-based Invertible Neural Networks Are Universal
  Diffeomorphism Approximators.
\newblock In \emph{Advances in Neural Information Processing Systems
  (NeurIPS)}.

\bibitem[{Toth et~al.(2019)Toth, Rezende, Jaegle, Racani{\`{e}}re, Botev, and
  Higgins}]{Toth2019-ix}
Toth, P.; Rezende, D.~J.; Jaegle, A.; Racani{\`{e}}re, S.; Botev, A.; and
  Higgins, I. 2019.
\newblock {Hamiltonian Generative Networks}.
\newblock In \emph{International Conference on Learning Representations
  (ICLR)}.

\bibitem[{Wheeler, Murray, and Schaefer(1993)}]{Wheeler1993-ni}
Wheeler, A.~A.; Murray, B.~T.; and Schaefer, R.~J. 1993.
\newblock Computation of dendrites using a phase field model.
\newblock \emph{Physica D}, 66(1): 243--262.

\bibitem[{Xiong et~al.(2021)Xiong, Tong, He, Yang, Yang, and
  Zhu}]{Xiong2020-ci}
Xiong, S.; Tong, Y.; He, X.; Yang, C.; Yang, S.; and Zhu, B. 2021.
\newblock {Nonseparable Symplectic Neural Networks}.
\newblock In \emph{International Conference on Learning Representations
  (ICLR)}.

\bibitem[{Zabusky and Kruskal(1965)}]{Zabusky1965-mn}
Zabusky, N.~J.; and Kruskal, M.~D. 1965.
\newblock Interaction of ``Solitons'' in a Collisionless Plasma and the
  Recurrence of Initial States.
\newblock \emph{Phys. Rev. Lett.}, 15(6): 240--243.

\bibitem[{Zhong, Dey, and Chakraborty(2020)}]{Zhong2019-rl}
Zhong, Y.~D.; Dey, B.; and Chakraborty, A. 2020.
\newblock {Symplectic ODE-Net: Learning Hamiltonian Dynamics with Control}.
\newblock In \emph{ICLR}.

\end{thebibliography}


%
%

\clearpage

\appendix

\section{Energy-Based Physical Models}\label{sec:app_general_models}
Equation \eqref{eq:target} [and its coordinate transformation \eqref{eq:transformed}] and the related model \eqref{eq:model} [and its coordinate transformation \eqref{eq:transformed_model}] describe not only Hamiltonian systems but also various types of other phenomena.
In fact, as shown in the theorem below, models similar to \eqref{eq:target},
\begin{align}\label{eq:model-g}
	\frac{\d u}{\d t} = G \frac{\partial H}{\partial u},
\end{align}
where $G$ is a skew-symmetric or negative-semidefinite matrix, have not only the energy conservation property but also the energy-dissipation property, depending on the matrix $G$. The same result for the models similar to \eqref{eq:transformed} is obtained in a straightforward way.
\begin{thm}\label{thm:10}
	Equation (12)  
	admits the energy conservation law, 
	\begin{align*}
		\frac{\d H}{\d t}
		= 0,
	\end{align*}
	if $G$ is skew-symmetric, and it admits the energy dissipation law,
	\begin{align*}
		\frac{\d H}{\d t}
		\leq 0,
	\end{align*}
	if $G$ is negative-semidefinite.
\end{thm}


\begin{proof}[Proof of Theorem \ref{thm:10}]
	We use the chain rule to obtain
	\begin{align}
		\frac{\d H}{\d t} = \frac{\partial H}{\partial u}^\top \frac{\d u}{\d t} = \frac{\partial H}{\partial u}^\top G \frac{\partial H}{\partial u}.
	\end{align}
	Hence, when the matrix $G$ is skew-symmetric, the energy conservation law holds:
	\begin{align*}
		\frac{\d H}{\d t}= 0
	\end{align*}
	because $v^\top G v = 0$ holds for all vectors $v$ when $G$ is skew-symmetric.

	When the matrix $G$ is negative semidefinite, the energy function is monotonically non-increasing:
	\begin{align}
		\frac{\d H}{\d t} =  \frac{\partial H}{\partial u}^\top G  \frac{\partial H}{\partial u} \leq 0.
	\end{align}
\end{proof}

As mentioned above, equations of this form are used not only in Hamiltonian mechanics but also in many fields of mathematical modeling, for example, in the Landau theory and the phase-field method. Phenomena such as phase separation, crystal growth, and crack propagation are modeled using these theories~(e.g., \citet{Caginalp1986-ur, Miehe2010-jh, Steinbach2009-uy, Wheeler1993-ni}).
The above model also includes the semi-discretized partial differential equations. For example, the KdV equation, which is a model of shallow-water waves,
\begin{align*}
	\frac{\partial u}{\partial t} = \alpha u \frac{\partial u}{\partial x} + \beta \frac{\partial^3 u}{\partial x^3},
\end{align*}
where $\alpha, \beta$ are parameters. This equation can be written as
\begin{align}
	\frac{\partial u}{\partial t} = \frac{\partial}{\partial x} \frac{\delta H}{\delta u},\label{eq:hamiltonianpde}
\end{align}
where $H(u, \partial u/\partial x)$ is the Hamiltonian and $\delta H/\delta u$ is the variational derivative of $H$, which is defined by
\begin{align*}
	\frac{\delta H}{\delta u} = \frac{\partial H}{\partial u} - \frac{\partial}{\partial x}\frac{\partial H}{\partial u_x},
\end{align*}
where $u_x$ denotes $\partial u/\partial x$. Equation \eqref{eq:hamiltonianpde} is an example of a Hamiltonian partial differential equation, in which $\partial/\partial x$ in front of $\delta H/\delta u$ plays the role of the skew-symmetric matrix $G$. In fact, if $\partial/\partial x$ is discretized by using a central difference operator,
\begin{align*}
&	\frac{\partial}{\partial x} \simeq 
\\
	& D:= \frac{1}{2 \Delta x}\begin{pmatrix}
		0      & 1      & 0      & \cdots & 0      & 0      & -1     \\
		-1     & 0      & 1      & \cdots & 0      & 0      & 0      \\
		0      & -1     & 0      & \cdots & 0      & 0      & 0      \\
		\vdots & \ddots & \ddots & \cdots & \ddots & \ddots & \vdots \\
		0      & 0      & 0      & \cdots & 0      & 1      & 0      \\
		0      & 0      & 0      & \cdots & -1     & 0      & 1      \\
		1      & 0      & 0      & \cdots & 0      & -1     & 0
	\end{pmatrix}.
\end{align*}
Because the difference matrix $D$ is skew-symmetric, the semi-discretized equation has the form of equation \eqref{eq:transformed_model} with a skew-symmetric matrix $G$. Similarly, an equation with the form
\begin{align}
	\frac{\partial u}{\partial t} = \frac{\partial^2}{\partial x^2} \frac{\delta H}{\delta u},\label{eq:dissipativepde}
\end{align}
is semi-discretized to equation \eqref{eq:transformed_model} with a negative-semidefinite $G$. In fact, if the $\partial^2/\partial x^2$ in front of $\delta H/\delta u$ is discretized,
\begin{align*}
&	\frac{\partial^2}{\partial x^2} 
  \simeq 
\\ 
	& D_2:= \frac{1}{\Delta x^2}\begin{pmatrix}
		-2     & 1      & 0      & \cdots & 0      & 0      & 1      \\
		1      & -2     & 1      & \cdots & 0      & 0      & 0      \\
		0      & 1      & -2     & \cdots & 0      & 0      & 0      \\
		\vdots & \ddots & \ddots & \cdots & \ddots & \ddots & \vdots \\
		0      & 0      & 0      & \cdots & -2     & 1      & 0      \\
		0      & 0      & 0      & \cdots & 1      & -2     & 1      \\
		1      & 0      & 0      & \cdots & 0      & 1      & -2
	\end{pmatrix},
\end{align*}
which is known to be negative semidefinite~\cite{Furihata2010}.

The universal approximatio-genn properties and the generalization error analysis are also applied to this model in a straightforward way under the assumption that $S$ is non-degenerate. 
In fact, the properties of the matrix $S$ used in the theoretical analyses are (1) $S$ is non-degenerate and (2) the norm of $S$ is bounded. If $S$ is a constant matrix, (2) is automatically satisfied and hence (1) is the only assumption of the matrix $S$.
In a part of the theoretical analysis, we consider the errors in the gradient $\nabla H$ of the energy function. However, typically, the loss function of $S \nabla H$ is minimized in the training process, and if $S$ is degenerate, errors in $\nabla H$ may not be small even when those in $S \nabla H$ are small. Hence, $S$ should be non-degenerate.


\section{Proof of Theorem \ref{thm:1}}\label{sec:app_proof1}
In this section we provide a proof of Theorem \ref{thm:1}. More specifically, we prove the following universal approximation theorem for the general energy-based physical models.
\begin{thm}\label{thm:1-gen}
	Let $H: \R^N \to \R$ be an energy function with the equation
	\begin{align*}
		\frac{\d u}{\d t} = G \frac{\partial H}{\partial u},
	\end{align*}
	where $u: t \in \R \mapsto u(t) \in \R^N$ and $G$ is a non-degenerate $N \times N$ matrix. Suppose that the state space $K$ of this system is compact and the right-hand side $G {\partial H}/{\partial u}$ is Lipschitz continuous. If the activation function $\sigma \neq 0$ belongs to $S^1_2(\R)$, then for any $\varepsilon > 0$ there exists a neural network $\hnn$ for which
	\begin{align*}
		\left\| G \frac{\partial H}{\partial u} - G \frac{\partial \hnn}{\partial u} \right\|_2 < \varepsilon
	\end{align*}
	holds. In addition, if the energy function is $C^\infty$, the function can be approximated by a $C^\infty$ neural network provided that the activation function is sufficiently smooth.
\end{thm}

To prove this theorem, we use the following theorem and the lemma, both of which were shown   
in \citet{Hornik1990-cg}.
\begin{thm}[Hornik et al., 1990]\label{thm:hornik}
	Let $\Sigma(\sigma)$ be the space of the neural networks with the activation function $\sigma$. 
	If the activation function $\sigma \neq 0$ belongs to $S^m_p(\R, \lambda)$ for an integer $m \geq 0$, then $\Sigma(\sigma)$ is $m$-uniformly dense in $C^{\infty}(K)$, where $K$ is any compact subset of $\R^N$.
\end{thm}
\begin{lem}[Hornik et al., 1990]\label{lem:hornik}
	Under the same assumption,
	$\Sigma(\sigma)$ is also dense in $S^m_p(\R, \lambda)$.
\end{lem}
Hence, if the activation function $\sigma$ of the hidden layer is in $S^m_p(\R, \lambda)$ and does not vanish everywhere, then for any sufficiently smooth function, there exists a neural network that approximates the function and its derivatives up to the order $m$ arbitrarily well on compact sets. This theorem has also been extended to the functions of multiple outputs; see~\citet{Hornik1990-cg}.
\begin{proof}[Proof of Theorem \ref{thm:1}]
	Because the target equation is determined only by the gradient of $H$, any function obtained by shifting $H$ by a constant gives the same equation. Hence, we choose and fix an energy function $H$ that yields the target equation.
	Because  $G {\partial H}/{\partial u}$ is Lipschitz continuous and hence continuous on the phase space $K$, this function is bounded and square-integrable. Thus, $G {\partial H}/{\partial u} \in S^0_2(K)$, which means $H$ is in $S^1_2(K)$. Therefore, from Lemma \ref{lem:hornik} and the assumption that the activation function is in $S^1_2(\R)$, for each $\varepsilon$, there exists a neural network that approximates $H$ in $S^1_2(K)$:
	\begin{align*}
		\left\| H - \hnn\right\|^2_2 + \left\|\frac{\partial H}{\partial u} - \frac{\partial \hnn}{\partial u} \right\|_2^2 < \frac{\varepsilon^2}{\| G\|^2_2},
	\end{align*}
	which gives
	\begin{align*}
		\left\| G \frac{\partial H}{\partial u} - G \frac{\partial \hnn}{\partial u} \right\|_2^2
		\leq \|G\|^2_2 \left\| \frac{\partial H}{\partial u} - \frac{\partial \hnn}{\partial u} \right\|_2^2
		< \varepsilon^2.
	\end{align*}
\end{proof}

\section{Proofs of Other Theorems}\label{sec:app_proof2}
\paragraph{Proof of Theorem 5}
\begin{proof}
	By substituting the equation, we obtain
	\begin{align*}
		\frac{\d H_\NN}{\d t} & = \frac{\partial H_\NN}{\partial x}^\top \frac{\d x}{\d t}
		\\ &
		= \frac{\partial H_\NN}{\partial x}^\top \frac{\partial u_\NN}{\partial x}^{-1} S \frac{\partial u_\NN}{\partial x}^{-\top} \frac{\partial H_\NN}{\partial x}
		 = 0
	\end{align*}
	because $S$ is skew-symmetric and hence for any vector $v$, $v^\top S v = 0$.
\end{proof}

\paragraph{Proof of Theorem \ref{thm:2}}
We prove the following theorem, which is a generalization of Theorem \ref{thm:2}.
\begin{thm}\label{thm:2-gen}
	Let $H: \R^N \to \R$ be an energy function for the equation
	\begin{align*}
		\frac{\mathrm{d}x}{\mathrm{d}t}=(\frac{\partial u}{\partial x})^{-1} G (\frac{\partial u}{\partial x})^{-\top} \frac{\partial H}{\partial x},
	\end{align*}
	where $x: t \in \R \mapsto x(t) \in \R^N$, $u: x \in \R^N \mapsto u(x) \in \R^N$, and $S$ is an $N \times N$ non-degenerate matrix.
	Suppose that the phase space $K$ of this system is compact and the right-hand side ${\partial H}/{\partial u}$ is Lipschitz continuous. Suppose also that $u$ is a $C^1$-diffeomorphism. If the functions $\sigma \neq 0$ and $\rho \neq 0$ belong to $S^1_2(\R)$, then for any $\varepsilon > 0$ there exist neural networks $\hnn$ with the activation functions $\sigma$ and $\unn$ with $\rho$ for which
	\begin{multline*}
		\left\|
		(\frac{\partial u}{\partial x})^{-1} G (\frac{\partial u}{\partial x})^{-\top} \frac{\partial H}{\partial x} \right.
\\
		-
		\left. (\frac{\partial \unn}{\partial x})^{-1} G (\frac{\partial \unn}{\partial x})^{-\top} \frac{\partial \hnn}{\partial x}
		\right\|_2 
		< \varepsilon
	\end{multline*}
	holds.
\end{thm}

\begin{proof}
	We need to prove the approximation property for $(\partial u/\partial x)^{-1}$. From the assumption that $\rho \neq 0$ is in $S^1_2(\R)$, there exists a function $\unn$ that approximates $\partial u/\partial x$. Because the determinant function of matrices is continuous, it is deduced that $\det \partial \unn/\partial x \neq 0$ and hence $(\partial \unn/\partial x)^{-1}$ exists. Because the matrix inverse is also continuous, $(\partial \unn/\partial x)^{-1}$ is also approximated by $\unn$.
\end{proof}

\section{Details on Numerical Experiments}\label{sec:app_ZK}
\subsection{Energy-based physical model with the coordinate transformations}
In this study, we considered a model with coordinate transformations:
\begin{align*}
	\frac{\mathrm{d}x}{\mathrm{d}t}=(\frac{\partial u}{\partial x})^{-1} G (\frac{\partial u}{\partial x})^{-\top} \frac{\partial H}{\partial x}.
\end{align*}
For a better understanding of the significance of the theoretical results presented herein, we performed numerical tests using this model. 

\paragraph{Double pendulum.}
First, we consider the double pendulum shown in Figure 2,
for which the equation of motion is
\begin{align*}
	 & \frac{\d \theta_1}{\d t} = \phi_1, \quad
	\frac{\d \theta_2}{\d t} = \phi_2,  \notag
	\\
	 & \frac{\d \phi_1}{\d t} = 
	 \Bigl(
	 G[\sin \theta_2 \sin(\theta_1-\theta_2) - \frac{m_1 + m_2}{m_2} \sin(\theta_1)]
	 \\ & \qquad \qquad
		- [l_1 \theta_1^2 \cos(\theta_1-\theta_2) + l_2 \theta_2^2] \sin(\theta_1 - \theta_2)
		\Bigr)
		\\ & \qquad \qquad 
		/
		{{l_1 [\frac{m_1 + m_2}{m_2} - \cos^2(\theta_1 - \theta_2)]}},
	\notag
	\\
	 & \frac{\d \phi_2}{\d t} =
	\Bigl(
	\frac{g(m_1 + m_2)}{m_2} [\sin \theta_1 \cos(\theta_1-\theta_2) - \sin(\theta_2)]
	\\ & \qquad 
		- [\frac{l_1 (m_1 + m_2)}{m_2} \theta_1^2  + l_2 \theta_2^2 \cos(\theta_1-\theta_2)] \sin(\theta_1 - \theta_2) 
		\Bigr)
		\\ & \qquad \qquad
		/
		{l_2 [\frac{m_1 + m_2}{m_2} - \cos^2(\theta_1 - \theta_2)]}. 
\end{align*}
The energy function of this system is
\begin{align*}
	H & =  \frac{1}{2}(m_1+m_2)l_1^2 \phi_1^2 
	\\ &
	+ \frac{1}{2}m_2 l_2^2 \phi_2^2
	+ m_2 l_1 l_2 \phi_1 \phi_2 \cos(\theta_1 - \theta_2)
	\\ &
	+ g (m_1 + m_2) l_1 \cos \theta_1 + g m_2 l_2 \cos(\theta_2),
\end{align*}
and the Lagrangian is
\begin{align*}
	\mathcal{L} = & \frac{1}{2}(m_1+m_2)l_1^2 \phi_1^2 + \frac{1}{2}m_2 l_2^2 \phi_2^2
	\\ &
	+ m_2 l_1 l_2 \phi_1 \phi_2 \cos(\theta_1 - \theta_2)
	\\ &
	- g (m_1 + m_2) l_1 \cos \theta_1 - g m_2 l_2 \cos(\theta_2),
\end{align*}
which gives the generalized momenta of this system:
\begin{align*}
	p_1 & = \frac{\partial \mathcal{L}}{\partial \phi_1}
	= (m_1 + m_2) l_1^2 \phi_1 + m_2 l_1 l_2 \phi_2 \cos(\theta_1 - \theta_2), \\
	p_2 & = \frac{\partial \mathcal{L}}{\partial \phi_2}
	= m_2 l_2^2 \phi_2 + m_2 l_1 l_2 \phi_1 \cos(\theta_1 - \theta_2).
\end{align*}
Because the generalized momenta are difficult to observe, we suppose that the values of the state variables $\theta_1, \theta_2$ and their derivatives $\phi_1, \phi_2$ are given as data.
In such a situation, HNNs \citep{Greydanus2019-gy}
\begin{align*}
	\frac{\d}{\d t}
	\begin{pmatrix}
		q_1 \\ q_2 \\ p_1 \\ p_2
	\end{pmatrix}
	= \begin{pmatrix}
		O  & I \\
		-I & O
	\end{pmatrix}
	\nabla \hnn(q_1, q_2, p_1, p_2)
\end{align*}
are not applicable, because this model can be used only when $p_1$ and $p_2$ are generalized momenta; in the case considered here, they are supposed to be unknown.

To confirm this, we tested the model
\begin{align*}
	\frac{\mathrm{d}x}{\mathrm{d}t}=(\frac{\partial u}{\partial x})^{-1} S (\frac{\partial u}{\partial x})^{-\top} \frac{\partial \hnn}{\partial x}
\end{align*}
and a naive HNN
\begin{align}\label{eq:hnn_naive_model}
	\frac{\d}{\d t}
	\begin{pmatrix}
		q_1 \\ q_2 \\ v_1 \\ v_2
	\end{pmatrix}
	= \begin{pmatrix}
		O  & I \\
		-I & O
	\end{pmatrix}
	\nabla \hnn(q_1, q_2, v_1, v_2),
\end{align}
where $v_1, v_2$ are not the generalized momenta but the velocities $v_1 = \dot{q}_1, v_2 = \dot{q}_2$.

The experiments were implemented using Python 3.8.5 with the packages PyTorch 1.7.1, NumPy 1.19.5, SciPy 1.6.0, autograd 1.3, and torchdiffeq 0.2.1.
For data, we used numerical solutions to the equation of motion 
with the parameters $l_1 = l_2 = 1.0$, $m_1 =1$, $m_2 = 2$, and $g = 9.8$
solved by SciPy odeint with 100 initial conditions randomly generated from the standard normal distribution.
We numerically integrated each orbit on the time interval $[0, 5]$, in which the computed values are evaluated at 100 points with an identical sampling rate. Then, the target data
\begin{align*}
	\frac{\d}{\d t}
	\begin{pmatrix}
		\theta_1 \\ \theta_2 \\ \phi_1 \\ \phi_2
	\end{pmatrix}
\end{align*}
are obtained by substituting the computed state variables into the right-hand side of the equation. 
The experiments were performed on an NVIDIA GeForce 2080 Ti.
The energy function was modeled by using a two-layered fully connected neural network with 50 hidden units and $\tanh$ as the activation function.
We minimized the mean squared error of the time derivative as the loss function using the Adam optimizer with a learning rate of $0.001$ and a batch-size of 200 for 10,000 iterations.

\begin{rem}
	For the coordinate transformations in the model, 
	invertible neural networks should be used; however, we used a simple neural network here because this model is assumed in the theorem in the main text.
	In the theorem, we needed neural networks that can approximate given functions and their derivatives.
	Although the universal approximation property of invertible neural networks was recently proved \citep{teshima2020couplingbased}, this theorem shows the universal approximation property of the function values only, not of the derivatives, which are needed for computation of the Jacobi matrix in our model.
\end{rem}

Examples of the predicted orbits are shown in Figure \ref{fig:double_pendulum_results}. 
The training losses were 13.6 for the HNN and 0.280 for the model with coordinate transformations.
As shown in this figure, the naive model failed to capture the dynamics correctly. This is because the dynamics of $\theta_1, \theta_2, \phi_1$, and$ \phi_2$ cannot be described by equation \eqref{eq:hnn_naive_model}. This illustrates that, to model physical phenomena by using a model of the form
\begin{align*}
	\frac{\d x}{\d t} = S \frac{\partial H}{\partial x},
\end{align*}
the choice of the coordinate system is important.

\paragraph{Mass-spring system.}
\begin{figure}[t]
	\centering
	\includegraphics[width=0.6\linewidth]{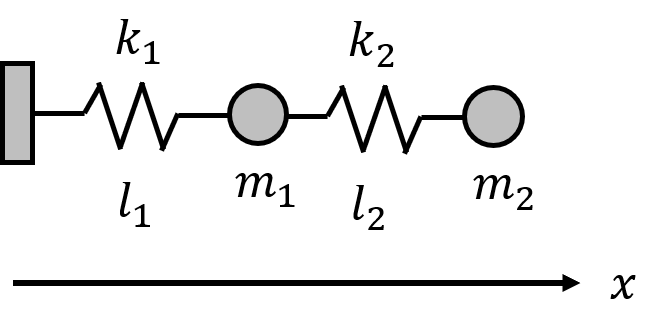}
	\caption{Target mass-spring system.}
	\label{fig:mass_spring}
\end{figure}

Due to the well-known chaotic behaviors of the double pendulum, the results, in particular the values of the losses, of the previous experiments are to a certain extent unstable, except for the Hamilton neural networks always failing.
Therefore, second, we investigated the models in more detail using the simple mass-spring system depicted in Figure \ref{fig:mass_spring}. The two mass points $m_1$ and $m_2$ are connected by springs, which respectively have the spring constants $k_1$ and $k_2$ and natural lengths $l_1$ and $l_2$. This system is a Hamiltonian system with the energy function
\begin{align}
	 & H(q_1, q_2, p_1, p_2)
	= \notag
	\frac{p_1^2}{2 m_1} + \frac{p_2^2}{2 m_2} 
	\\ &
	+ \frac{k_1 (q_1-l_1)^2}{2} + \frac{k_2 (q_2-q_1 - l_2)^2}{2}, \label{eq:mass_spring}
\end{align}
where $q_1$, $q_2$ are the positions of the mass points and $p_1$, $p_2$ are the momenta, which are defined by $p_1 = m_1 v_1, p_2 = m_2 v_2, v_1 = \d p_1/\d t, and v_2 = \d p_2/\d t$.
Suppose that we do not know the exact values of $m_1$ and $m_2$ and only the positions $q_1$ and $q_2$ and their derivatives can be observed.
Although $m_1$ and $m_2$ may be estimated from the data, for evaluation of the models, we tried to model the dynamics only using $q_1$, $q_2$, and their derivatives.

We examined the naive model,
the model with the coordinate transformation, and the neural ordinary differential equation (ODE,  \citet{Chen2018-yw})
\begin{align*}
	\frac{\d}{\d t}
	\begin{pmatrix}
		q_1 \\ q_2 \\ v_1 \\ v_2
	\end{pmatrix}
	= f_{\mathrm{NN}}(q_1, q_2, v_1, v_2).
\end{align*}

The experimental conditions were almost the same as those in the previous experiment, except for the batch size, which we set to 100. For data, we used numerical solutions to
\begin{align}\label{eq:mass_spring2}
	\frac{\d}{\d t}
	\begin{pmatrix}
		q_1 \\ q_2 \\ v_1 \\ v_2
	\end{pmatrix}
	=
	\begin{pmatrix}
		v_1                                                               \\
		v_2                                                               \\
		- \frac{k_1}{m_1} (q_1 - l_1) + \frac{k_2}{m_1} (q_2 - q_1 - l_2) \\
		- \frac{k_2}{m_2} (q_2 - q_1 - l_2)                               \\
	\end{pmatrix},
\end{align}
which is equivalent to equation \eqref{eq:mass_spring}.


Examples of the predicted orbits are shown in Figure \ref{fig:mass_spring_results}. While the model with coordinate transformations produced an orbit that was nearly identical to the ground truth, the naive model failed to predict the states.
In fact, it should be impossible to rewrite equation \eqref{eq:mass_spring2} as \eqref{eq:hnn_naive_model} with a certain energy function;
for example, when the system has just one mass point and the equation of motion is given by
\begin{align*}
	\frac{\d}{\d t}
	\begin{pmatrix}
		q_1 \\ v_1
	\end{pmatrix}
	=
	\begin{pmatrix}
		v_1                           \\
		- \frac{k_1}{m_1} (q_1 - l_2) \\
	\end{pmatrix},
\end{align*}
this can be transformed into a Hamiltonian system
\begin{align*}
	 & \frac{\d}{\d t}
	\begin{pmatrix}
		q_1 \\ v_1
	\end{pmatrix}
	=
	\begin{pmatrix}
		0  & 1 \\
		-1 & 0
	\end{pmatrix}
	\nabla \tilde{H},                                                    \\
	 & \tilde{H} = \frac{v_1^2}{2 m_1^\prime} + k_1^\prime (q_1 - l_1)^2
\end{align*}
with $k_1^\prime = k_1/m_1, m_1^\prime$ and the mass $m_1^\prime$ = 1. Hence, for this system, the HNNs are applicable without knowledge of $m_1$. However, for system \eqref{eq:mass_spring2}, such a transformation cannot be applied.
Meanwhile, in the model with the coordinate transformation, a coordinate transformation that makes the equation Hamiltonian is explored.

\begin{figure}[t]
	\centering
	\subfigure[Ground truth]{\includegraphics[width=0.48\linewidth]{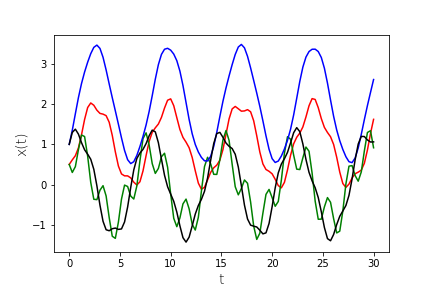}}
	\subfigure[Neural ODE]{\includegraphics[width=0.48\linewidth]{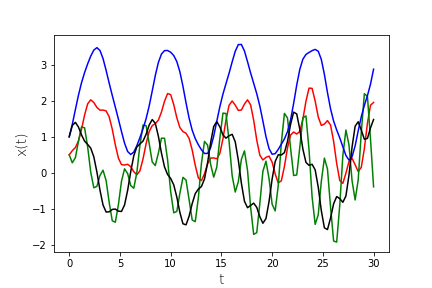}}\\[-2mm]
	\subfigure[Naive  HNN]{\includegraphics[width=0.48\linewidth]{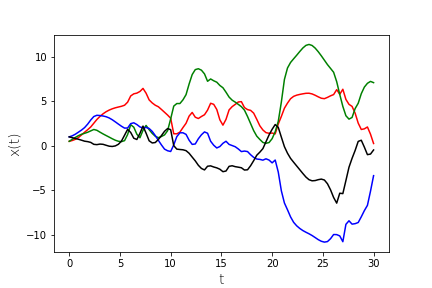}}
	\subfigure[Model with coordinate transformations]{\includegraphics[width=0.48\linewidth]{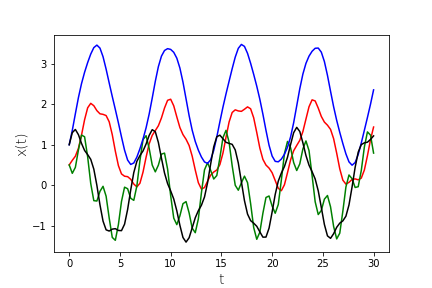}}
	\caption{Example predicted orbits, where red = $q_1$, green = $v_1$, blue = $q_2$, and black = $v_2$.}
	\label{fig:mass_spring_results}
	\centering
	\subfigure[Neural ODE]{\includegraphics[width=0.48\linewidth]{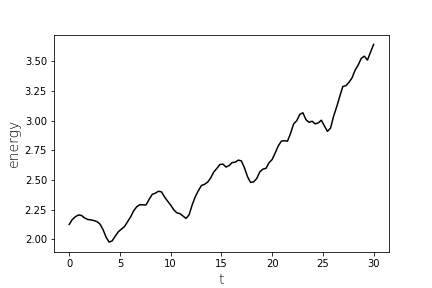}}\
	\subfigure[Model with coordinate transformations]{\includegraphics[width=0.48\linewidth]{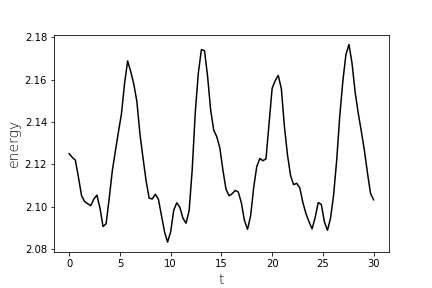}}\\
	\caption{Predicted values of the energy function.}
	\label{fig:energy}
\end{figure}

The result obtained by the neural ODE is better than that by the naive model; however, the prediction error became larger as time increased. 
This is due to the non-existence of the energy conservation law for the neural ODE. In fact, the value of the energy function of the neural ODE model was increasing, as shown in Figure \ref{fig:energy}.

The losses for the above models are listed in Table \ref{tab:mass_spring_results} for 12 experiments for each model. 
The values of the loss functions of the neural ODE were very small and those of the model with coordinate transformations were not stable. Actually, the histories of the loss functions for this model were not monotonic at all. This is due to the non-uniqueness of the model with coordinate transformations; if a model is fit to the given data quite well, then so are its canonical transformations. Due to this feature of the model, after the loss function decreased, it sometimes increased significantly in the search for different coordinate transformations. The failure in the learning process can be detected as the large training loss, and if the learning process is successful, the test loss is competitive with that of the neural ODEs. Moreover, the neural ODEs exhibit energy drift, as seen above. Therefore, the model with coordinate transformations is the most suitable for a long-term  physics simulation in an unknown coordinate system.


\begin{table*}[t]
	\centering
	\caption{Training and test losses for 12 trials.}
	\begin{tabular}{|l|l|l|l|l|l|}
		\hline
		\multicolumn{2}{|l|}{Neural ODE} & \multicolumn{2}{l|}{HNN} & \multicolumn{2}{l|}{\parbox{11em}{Model with                                                                     \\ coordinate transformations\\[-0.6em]}}
		\\ \hline
		\parbox{5em}{train}              & \parbox{5em}{test}       & \parbox{5em}{train}                          & \parbox{5em}{test} & \parbox{5.5em}{train} & \parbox{5.5em}{test} \\ \hline
		0.000361                         & 0.00223                  & 0.566                                        & 15.6               & 0.00697               & 0.00823              \\
		0.000221                         & 0.00101                  & 0.614                                        & 18.1               & 17.9                  & 3214.7               \\
		0.0000209                        & 0.000605                 & 0.630                                        & 17.0               & 0.00198               & 0.00688              \\
		0.000197                         & 0.000612                 & 0.527                                        & 17.4               & 0.000253              & 0.00254              \\
		0.000128                         & 0.000552                 & 0.637                                        & 15.9               & 24.2                  & 3247.7               \\
		0.000258                         & 0.00104                  & 0.603                                        & 17.3               & 0.000954              & 0.0101               \\
		0.000999                         & 0.00150                  & 0.586                                        & 16.2               & 0.000550              & 0.00996              \\
		0.000527                         & 0.00181                  & 0.583                                        & 18.9               & 0.230                 & 2.96                 \\
		0.000307                         & 0.00164                  & 0.572                                        & 19.6               & 0.000565              & 0.00513              \\
		0.0000450                        & 0.000742                 & 0.599                                        & 18.2               & 0.00316               & 0.00722              \\
		0.0000522                        & 0.00101                  & 0.627                                        & 17.9               & 0.000615              & 0.00401              \\
		0.0000374                        & 0.000931                 & 0.614                                        & 16.1               & 0.00909               & 0.229
		\\ \hline
	\end{tabular}
	\label{tab:mass_spring_results}
\end{table*}

\end{document}